\documentclass[11pt,fleqn]{amsart} 

\usepackage[usenames,dvipsnames]{xcolor}

\usepackage{amsthm}
\usepackage{amsfonts}
\usepackage[english]{babel}
\usepackage[usenames]{xcolor}
\usepackage{graphicx}
\usepackage{soul}
\usepackage{stfloats}
\usepackage{morefloats}
\usepackage{cite}
\usepackage{lscape}
\usepackage{epstopdf}
\usepackage{braket}
\usepackage[lite]{amsrefs}
\usepackage{mathbbol}
\usepackage{tikz}
\usepackage{tikz-cd}
\usepackage{mathtools}
\usepackage{shuffle}

\usepackage{amsmath,amstext,amsopn,amsfonts,eucal,amssymb}
\usepackage{graphicx,wrapfig,url}

\newcommand\Z{{\mathbb Z}}

\newcommand\R{{\mathbb R}}

\DeclareMathOperator{\tr}{Tr}

\newtheorem{theorem}{Theorem}[section]
\newtheorem{corollary}[theorem]{Corollary}
\newtheorem{lemma}[theorem]{Lemma}
\newtheorem{proposition}[theorem]{Proposition}
\newtheorem{definition}[theorem]{Definition}

\newtheorem{example}[theorem]{Example}

\newtheorem{remark}[theorem]{Remark}

\newcommand{\cal}{\mathcal}

\usepackage{soul}

\begin{document}

\title{Quantum Cocycle Invariants of Knots
from Yang-Baxter Cohomology} 

\author{Masahico Saito} 
\address{Department of Mathematics and Statistics, 
	University of South Florida, Tampa, FL 33620} 
\email{saito@usf.edu} 

\author{Emanuele Zappala} 
\address{Department of Mathematics and Statistics, Idaho State University\\
	Physical Science Complex |  921 S. 8th Ave., Stop 8085 | Pocatello, ID 83209} 
\email{emanuelezappala@isu.edu}

\maketitle

\begin{abstract}
Yang-Baxter operators (YBOs) have been employed to construct quantum knot invariants. More recently, cohomology theories for YBOs have been independently developed, drawing inspiration from analogous theories for quandles and other discrete algebraic structures. Quandle cohomology, in particular, gives rise to cocycle invariants for knots via $2$-cocycles, which are closely related to quandle extensions,  and knotted surfaces via $3$-cocycles. These quandle cocycle knot invariants have also been shown to admit interpretations as quantum invariants.

Similarly, $2$-cocycles in Yang-Baxter cohomology can be interpreted in terms of deformations of YBOs. Building on these parallels, we introduce quantum cocycle invariants of knots using $2$-cocycles 
of Yang-Baxter cohomology, from the perspective of deformation theory. In particular, we demonstrate that the quandle cocycle invariant can be interpreted in this framework, while the quantum version yields stronger invariants in certain examples. 

We develop our theory along two primary approaches to quantum invariants: via trace constructions and through (co)pairings. Explicit examples are provided, including those based on the Kauffman bracket. Furthermore, we show that both the Jones and Alexander polynomials can be derived within this framework as invariants arising from higher-order formal Laurent polynomial deformations via Yang-Baxter cohomology.
		
\end{abstract}

\date{\empty}

\tableofcontents

\section{Introduction}

The Jones polynomial was originally introduced using von Neumann algebras~\cite{Jones}, and this was later reformulated using Yang-Baxter operators (YBOs) and 
combinatorial techniques on knot diagrams \cite{Kauffman}. 
This discovery has led to a new area of quantum topology.

Quandles are  discrete self-distributive  systems with  additional conditions~\cite{Joyce,Matveev},  
that have been introduced before the Jones polynomial. 
Quandle cocycle invariants of knots and knotted surfaces have been defined since then, 
and brought a new light to quandles~\cite{CJKLS}.
	Quandle cocycle invariants were  defined using 2-cocycles of quandle cohomology theory, in a form of state sum
with 2-cocycles as weights, though the states are limited to  quandle colorings.
Applications to various properties of knots and knotted surfaces followed.
Algebraic aspects of quandle homology theory such as extensions have also been explored~\cite{CENS},
in parallel to group cohomology theory. 
It was shown that they can be interpreted as quantum invariants,  i.e. as the trace of suitable operators associated to a closed braid form of a knot~\cite{Grana}.

Quandle homology theory has been generalized to homology theory of set-theoretic Yang-Baxter operators 
\cite{CES}, and further to general Yang-Baxter operators \cite{Eis,Eis1,Lebed,Jozef}.
Homology and cohomology theories for categorical self-distributivity are also developed \cite{cat_SD,coh_deform}. 
Another cohomology theory for YBOs was defined in \cite{SZ-YBH} in such a way that it is unified with Hochschild cohomology, and that follows deformation theory of  braided algebras. This  point of view from deformation theory parallels the fact that quandle 2-cocycles can be used to define extensions of quandles.

Quandles give rise to  quandle coloring invariants, 
and the linearization of quandles give YBOs, which  provide quantum knot invariants.
Cocycles of quandle cohomology theory are used to define  quandle cocycle invariants.
Since there exist Yang-Baxter cohomology theories as mentioned above, it is natural to ask whether their cocycles give rise to knot  invariants. 
This is the goal of the current paper, and a motivation to call  the resulting invariant {\it quantum cocycle invariant}.
These conceptual considerations are schematically depicted in Figure~\ref{scheme}.

\begin{figure}[htb]
\begin{center}
\includegraphics[width=4in]{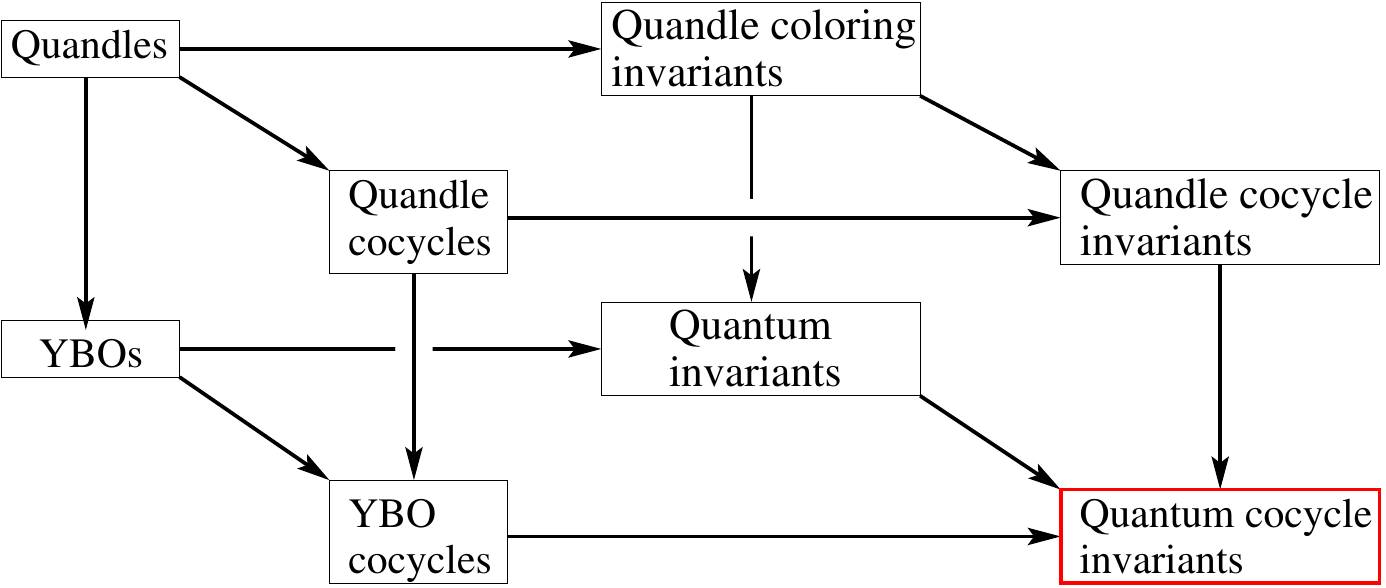}
\end{center}
\caption{}
\label{scheme}
\end{figure}

There are two main methods to construct knot invariants from YBOs, (1) by using closed braid form and taking trace,
and (2) fixing a height function and assign maps corresponding to local maxima and minima.
We seek both approaches. For (1), well-definedness requires to verify Markov trace property, that ensures invariance under Markov stabilizations, and for (2), maps assigned to maxima and minima need to satisfy additional Reidemeister moves 
involving them. 

Specifically, for both approaches, let $R$ be a given YBO that defines by braid group representation a knot invariant. Take a YBO 2-cocycle $\phi$ of the YB cohomology for $R$ and the deformation $R+\hbar \phi$ (see below). For (1) we take the trace and show that it gives rise to an invariant (Theorem~\ref{thm:inf_YB_inv}), and for (2) we use height function to show well definedness (Theorem~\ref{thm:cupcap}) and present requirements for the 2-cocycles to satisfy for well definedness.
For both approaches, we provide explicit examples of this construction starting from well known $R$, such as the tensorized version of quandle operations and the Kauffman bracket. 

	We  also  consider higher order deformations of YBOs 
	  later in Section~\ref{sec:higher_def},  where the YB $2$-cocycle is integrable, and construct associated invariants from YBOs of type $R+ \sum_{i=1}^n\hbar^i \phi_i$, where $\phi_i$ are cochains that clear the deformation obstruction in the third YB cohomology group. This result is formulated in Theorem~\ref{thm:higher}, 
	and is used in Section~\ref{sec:Jones_Alex}.  
	
In addition, the approach of generalizing quandle cocycle invariants by means of YB cohomology allows us to answer another natural question. Namely, whether there is an underlying common structure behind quandle cocycle invariants and the Alexander and Jones polynomials. We show that they are all interpreted as the YB cocycle invariants introduced and studied in this article. 

	In fact, we introduce a suitable generalization of the aforementioned higher order invariants to Laurent polynomial series, and then show that it is possible to find pre-YB operators, along with suitable cohomology classes and higher order deformations such that the Jones and Alexander polynomials are the Laurent quantum cocycle invaraints arising from them. The results are formalized in Theorem~\ref{thm:Jones} and Theorem~\ref{thm:Alex}. 

This paper also adds a new relation between algebraic deformation cohomology theories and knot theory.
Diagrammatic methods in knot theory have been applied to algebraic deformation theories. For example, 
diagrams are used to construct cohomology theories for categorical self-distributive structures in \cite{cat_SD}, and Yang-Baxter and Hochschild cohomology theories are unified with extensive use of knot and graph diagrams in \cite{SZ-YBH,SZ-BC}. In this paper, the opposite direction is proposed, namely applying deformation cohomology theories to the construction of knot invariants. Since a number of  knot invariants are defined using algebraic systems, the idea of using their deformation cohomology for constructions of new knot invariants 
may be applicable to such  invariants as well, therefore providing more general applicability of our current results.

This article is organized as follows. Section~\ref{sec:prelim} reviews the two approaches of defining quantum invariants as mentioned above.
In Section~\ref{sec:YB_coh}, 
we review  the  Yang-Baxter homology theory for  low  degrees 
with diagrammatic conventions,
and then describe a definition in  all higher degrees.   
The definition of the quantum cocycle invariant using the trace approach is given in Section~\ref{sec:inf_def}, following Turaev's formulation.
Concrete examples, including the tensorized quandle cocycle invariant, are presented in Section~\ref{sec:traces}.
The definition for the second approach using maxima and minima is given in Section~\ref{sec:cupcap}, and developed in Section~\ref{sec:bracket}
 for the  Kauffman bracket polynomial with diagrammatic computations.
Up to this section the  invariant uses infinitesimal deformation, that is of degree up to 1 with respect to 
the deformation parameter $\hbar$. 
The invariant is extended to higher order invariant in Section~\ref{sec:higher_def}, extending degrees with respect to $\hbar$.
In Section~\ref{sec:Jones_Alex}, the higher order invariant is further extended to negative degrees to allow Laurent polynomial expansions,
which provides an interpretation of Jones and Alexander polynomials as such quantum cocycle invariants.

Throughout the paper, the results apply to either a knot or a link, so that we use the convention that a knot $K$ means either a knot or a link.

\section{Preliminaries}\label{sec:prelim}

In this section we review two approaches for defining quantum knot invariants,
one via traces and the other via height functions and maxima/minima.
We follow these two approaches to construct quantum cocycle invariants in later sections.
 Unless otherwise specified, we use the notation $\mathbb k$ for a unital ring, and $V$ for a module over $\mathbb k$.

Recall that for an invertible map $R: V \otimes V \rightarrow V \otimes V$
of a ${\mathbb k}$-module $V$,
the equation
$$(R \otimes {\mathbb 1}) ({\mathbb 1} \otimes R) (R \otimes {\mathbb 1}) 
=({\mathbb 1} \otimes R) (R \otimes {\mathbb 1}) ({\mathbb 1} \otimes R) $$
is called the {\it Yang-Baxter} equation (YBE),  and an invertible solution $R$ of it is called a Yang-Baxter (YB) operator,
or R-matrix. We will also say YBO for short. 
YBE is diagrammatically represented as in Figure~\ref{YBE}, where $R$ is represented by a 
 crossing. Diagrams are read from top to bottom, and the rightmost term of  maps that acts first, corresponds to the top portion of a diagram. For example, the rightmost term $(R \otimes \mathbb 1)$ in the left-hand side of YBE above represents the top left crossing  with a straight line to its right, in the left-hand side of Figure~\ref{YBE}.

If $R$ satisfies the YBE but is not assumed to be invertible, then $R$ is called {\it pre-Yang-Baxter operator},
or pre-YBO for short.
If $R$ is invertible the inverse is denoted by $R^{-1}$, and to indicate that either $R$ or its inverse are used, we employ the notation $R^{\pm}$
for short. In Figure~\ref{YBE}, when all strings are oriented downwards, crossings are defined to be positive.
If $R$ corresponds to a positive crossing, $R^{-1}$ is represented by a negative crossing.

\begin{figure}[htb]
\begin{center}
\includegraphics[width=1in]{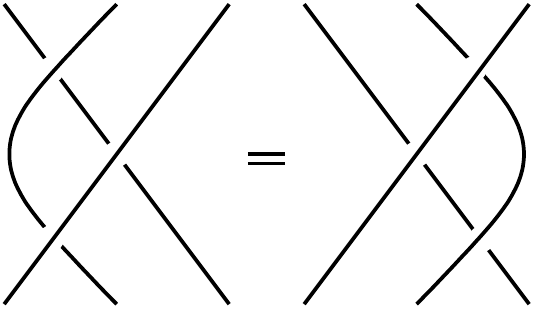}
\end{center}
\caption{}
\label{YBE}
\end{figure}

\subsection{Knot invariants from closed braids and traces}

There are two main constructions of quantum knot invariants from YBOs: one is through braid group representations and their trace maps that are invariant under Markov moves, and the other is by fixing a height function on the plane of projection and  incorporating maps corresponding to maxima/minima of knot diagrams.
 Although these two approaches are closely related,  
 we examine these approaches separately to take advantages of computations of known invariants and their deformations. 
In this subsection we review the first approach of trace, and in the next subsection we review the second maxima/minima approach.

Recall  that an {\it enhanced YBO} (EYBO) is a quadruple $S = (R, \alpha, \beta, \mu)$ consisting of a YBO $R$, two invertible scalars $\alpha, \beta \in \mathbb k^\times$, and a linear map $\mu: V\longrightarrow V$ satisfying the following two compatibility conditions
\begin{eqnarray}
		(\mu\otimes \mu) R &=& R (\mu\otimes \mu),\label{eqn:mu_comm}\\
		\tr_2(R^{\pm}(\mu\otimes \mu)) &=& \alpha^\pm\beta\mu, \label{eqn:partial_tr}
\end{eqnarray}
where $\tr_2$ indicates the partial trace of an operator with respect to its second tensorand~\cite{Tur}. To shorten notation, in this article we also indicate an EYBO simply by $R$, when no risk of confusion arises. Observe that when $\mu$ is invertible, Equation~\eqref{eqn:partial_tr} is equivalent to
\begin{eqnarray}
		\tr_2(R^{\pm}(\mathbb 1\otimes \mu)) = \alpha^{\pm}\beta \mathbb 1,
\end{eqnarray}
 where $\mathbb 1$ denotes the identity mapping. 
Using EYBOs, Turaev gave a definition of quantum invariants of knots and links as follows~\cite{Tur}.
Let $\mathbb B_m$ denote the $m$-string braid group, whose elements are represented by vertical string diagrams with crossings.  Recall that a YBO $R$ induces a representation of the braid group $\rho: \mathbb B_m \longrightarrow {\rm Hom}_{\mathbb k}(V^{\otimes m},V^{\otimes m})$, defined on generators as $\rho(\sigma^\pm_i) = \mathbb 1^{\otimes (i-1)}\otimes R^\pm \otimes \mathbb 1^{\otimes (m-i-1)}$. 

A knot or a link $K$ is represented as a closed braid form (Alexander's theorem), where 
the closure is to connect the top and bottom end points of a braid by parallel nested circular arcs to the right.
Let $b_m  \in \mathbb B_m$ denote a braid whose closure represents $K$. Diagrams in closed braid form are oriented downwards on the braid portion unless otherwise stated. 
Let $w(b_m)$ denote the writhe  of $b_m$, that is the number of positive crossings (positive braid generators) minus the negative ones. For an EYBO $S=(R, \alpha, \beta, \mu)$, a knot invariant $T_S(K)$ is defined by 

\begin{eqnarray}
		T_S(K) = \alpha^{w(b_m)} \beta^{-m}\tr(\rho(b_m) \mu^{\otimes m}). 
\end{eqnarray}

The well definedness is proved by showing its invariance under Markov moves,  which are conjugations and (de)stabilizations.

\subsection{Knot invariants using maxima/minima}
In this section we review a construction of knot invariants using maxima/minima of knot diagrams with a fixed height function (see, for example, \cite{K&P}).  Let $h: \R^2 \rightarrow \R$ be a projection onto a subspace that is in general position with respect to the image of a given knot, so that the critical points of the knot projection consists of isolated local maxima, minima and crossing points. Such a map $h$ is regarded as a height function on the plane with respect to a given knot diagram, and the definition of knot invariant involves (co)parings for maxima and minima, in addition to YBOs for crossings. 
 This approach is used for unoriented knots and their diagrams.

A pairing $\cup : V \otimes V \rightarrow \mathbb{k} $ and 
a copairing $\cap: {\mathbb k} \rightarrow V \otimes V$ in a  module $V$ over a unital ring $\mathbb{ k}$ are said to have (or satisfy) the {\it switchback property}
if they satisfy the equalities
$$(\cup \otimes {\mathbb 1}) ({\mathbb 1} \otimes \cap)={\mathbb 1} 
= ({\mathbb 1} \otimes \cup ) ( \cap  \otimes {\mathbb 1} ). $$

Let $R$ be a YBO. 
A pairing $\cup$ is said to satisfy (or have)  the {\it passcup}  property  with respect to $R$ if it satisfies 
\begin{eqnarray*}
({\mathbb 1} \otimes \cup)(R \otimes  {\mathbb 1}) = ( \cup \otimes  {\mathbb 1} )( {\mathbb 1}  \otimes R^{-1} ), \\
({\mathbb 1} \otimes \cup)(R^{-1} \otimes  {\mathbb 1}) = ( \cup \otimes  {\mathbb 1} )( {\mathbb 1}  \otimes R).
\end{eqnarray*}
Similarly a copairing $\cap$ is said to satisfy (or have) the {\it passcap}  property  with respect to $R$ if it satisfies 
\begin{eqnarray*}
(R \otimes  {\mathbb 1})  ({\mathbb 1} \otimes \cap)= ( {\mathbb 1}  \otimes R^{-1} ) ( \cap \otimes  {\mathbb 1} ), \\
(R^{-1} \otimes  {\mathbb 1})  ({\mathbb 1} \otimes \cap)= ( {\mathbb 1}  \otimes R) ( \cap \otimes  {\mathbb 1} ).
\end{eqnarray*}
In Figure~\ref{moves}, the switchback, passcup and passcap moves are depicted from left to right, respectively.

\begin{figure}[htb]
\begin{center}
\includegraphics[width=6in]{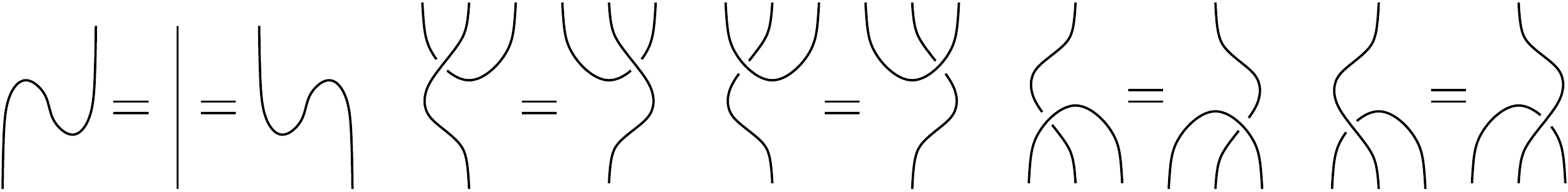}
\end{center}
\caption{}
\label{moves}
\end{figure}

A ${\mathbb k}$-module $V$ with a YBO and a pairing and copairing that satisfy switchback, passcup, and passcap properties provide a knot invariant up to regular isotopy equivalence relation without type I Reidemeister move
by interpreting a given knot diagram as a linear map from 
${\mathbb k}$ to ${\mathbb k}$, corresponding to empty to empty strings, and 
the linear map is a constant multiple, which defines  a desired knot invariant.
 See for example, \cite{K&P}.
If in addition the equation $\cup R=\cup$ and $R \cap =\cap$  
are satisfied, this procedure provides a knot invariant up to isotopy. 
Alternatively, one examines the effect of type I Reidemeister move to the regular isotopy invariant, and makes a normalization by the writhe of a given diagram when oriented.

\section{Yang-Baxter cohomology}\label{sec:YB_coh}

In this section we introduce a YB cohomology theory which differs in defining formulas from the ones studied in \cite{Eis1,Eis,Lebed,Jozef}. 
In particular, the main difference  from \cite{Eis1,Eis} lies in  the fact that the inverse of the YB operator is not used, and it can therefore be defined for pre-YB operators (i.e. not necessarily invertible solutions to the YBE).   We will use the notation ${\rm Hom}(V^{\otimes n}, V^{\otimes n})$ as a shorthand notation for ${\rm Hom}_{\mathbb k}(V^{\otimes n}, V^{\otimes n})$ throughout the section, and the rest of the article.

\subsection{YB cohomology in low  degrees}

In this subsection we review YB cohomology theory from \cite{SZ-YBH} in low degrees, as it is used in the several sections that follow.

The cochain groups are defined by 
$C^0_{\rm YB}(V,V)=0$ and 
$C^n_{\rm YB}(V,V)={\rm Hom} (V^{\otimes n},V^{\otimes n})$ for $n>0$.
We define the differentials for $f \in C^1_{\rm YB}(V,V) $
and $\phi \in C^2_{\rm YB}(V,V)$ by 
\begin{eqnarray*}
	\delta^1_{\rm YB} (f) 
	&=&
	R ( f \otimes {\mathbb 1}) + R ( {\mathbb 1} \otimes  f ) 
	-  ( f \otimes {\mathbb 1}) R - ( {\mathbb 1} \otimes  f )  R ,\\
	\delta^2_{\rm YB} (\phi) &=&
	(R \otimes {\mathbb 1} ) ( {\mathbb 1}  \otimes R ) ( \phi \otimes  {\mathbb 1} )
	+ (R \otimes {\mathbb 1} ) ( {\mathbb 1}  \otimes \phi ) ( R \otimes  {\mathbb 1} )
	+  (\phi \otimes {\mathbb 1} ) ( {\mathbb 1}  \otimes R ) ( R \otimes  {\mathbb 1} ) \\
	& & -  ( {\mathbb 1}  \otimes R ) ( R \otimes  {\mathbb 1} ) ( {\mathbb 1}  \otimes \phi ) 
	- ( {\mathbb 1}  \otimes R ) ( \phi \otimes  {\mathbb 1} ) ( {\mathbb 1}  \otimes  R ) 
	- ( {\mathbb 1}  \otimes \phi ) ( R \otimes  {\mathbb 1} ) ( {\mathbb 1}  \otimes  R ) .
\end{eqnarray*}

 When the  pre-YBO $R$ needs to be specified, we also use the notation $\delta^n_{R}$
for $\delta^n_{\rm YB}$. 
The fact that these define homology theory  is proved in the next section.
The corresponding cocycle, coboundary, homology groups are denoted as usual using $Z, B, H$, respectively.

\begin{figure}[htb]
\begin{center}
\includegraphics[width=2in]{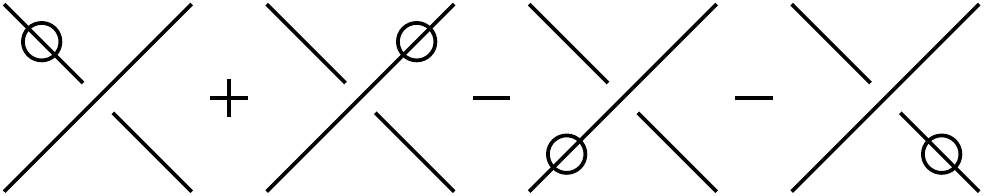}
\end{center}
\caption{}
\label{YBdiff1}
\end{figure}

\begin{figure}[htb]
\begin{center}
\includegraphics[width=3in]{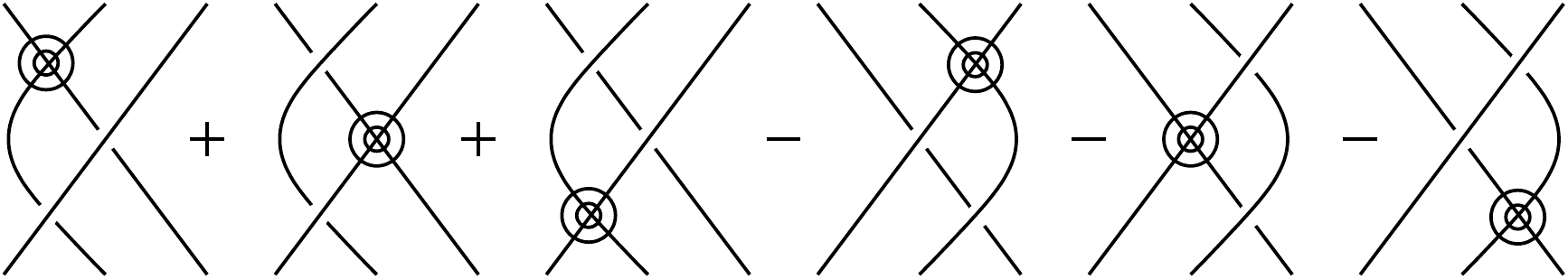}
\end{center}
\caption{}
\label{YBd2}
\end{figure}

The expression for $ \delta^1_{\rm YB} (f) $ is depicted in Figure~\ref{YBdiff1}, where YBOs are represented by 
crossings and   1-cochains  
$f$ are represented by  circles on edges.
The expression for $ \delta^2_{\rm YB} (\phi) $ is depicted in Figure~\ref{YBd2}, where 
  2-cochains  
$\phi$ are represented by double circles.
The first positive terms in the figure correspond to the left-hand side of the YBE, while the negative terms to the right-hand side.
The three terms for  each is obtained by replacing a crossing by a double circle at exactly one place.

We cite the following lemma that relates  2-cochains  
 to deformations.

\begin{lemma}[\cite{SZ-YBH}] \label{lem:YB2toR}
Let $R: V \rightarrow V $ be a YBO. 
Let $\phi \in C^2_{\rm YB} (V,V)$ be a YB 2-cochain for $R$.
Then, $\tilde R = R + \hbar \phi$ is a YBO on $\tilde V =  \mathbb k [[ \hbar ]] / (\hbar^2)  \otimes V $ if and only if  $\phi \in Z^2_{\rm YB} (V,V)$. 
\end{lemma}

\begin{proof}
One computes the YBE for $\tilde R$, and checks that the terms for $\hbar $ constitutes the YB 2-cocycle condition. 
\end{proof}

 This YBO  $\tilde R =  R + \hbar \phi$ on $\tilde V$ is called an {\it infinitesimal deformation } of $R$.

\subsection{ YB cohomology in higher  degrees  } 

  In this subsection we review definitions of YB cohomology theory from \cite{SZ-YBH},
see also \cite{SZ-BC}.

We set $\sigma_{n,i} = \mathbb 1^{\otimes (i-1)}\otimes R \otimes \mathbb 1^{\otimes(n-i-1)}$. Then, for $\phi \in {\rm Hom}(V^{\otimes n}, V^{\otimes n})$, the partial
 differentials for YB cohomology are defined as
\begin{eqnarray*}
	d_{{\rm YB},i}^n(\phi) &=& \sigma_{n+1,n+1-i}\cdots\sigma_{n+1,2}\sigma_{n+1,1}(\mathbb 1\otimes \phi)\sigma_{n+1,1}\cdots\sigma_{n+1,i-1}\sigma_{n+1,i}\\
	&& - \sigma_{n+1,i+1}\cdots\sigma_{n+1,n}\sigma_{n+1,n+1}(\phi \otimes \mathbb 1)\sigma_{n+1,n+1}\cdots\sigma_{n+1,i+2}\sigma_{n+1,i+1}.
\end{eqnarray*}
The partial differentials $d_{{\rm YB},i}^n$ are graphically depicted as in Figure~\ref{fig:partial_YB}.
\begin{figure}[htb]
	\begin{center}
		\includegraphics[width=2.8in]{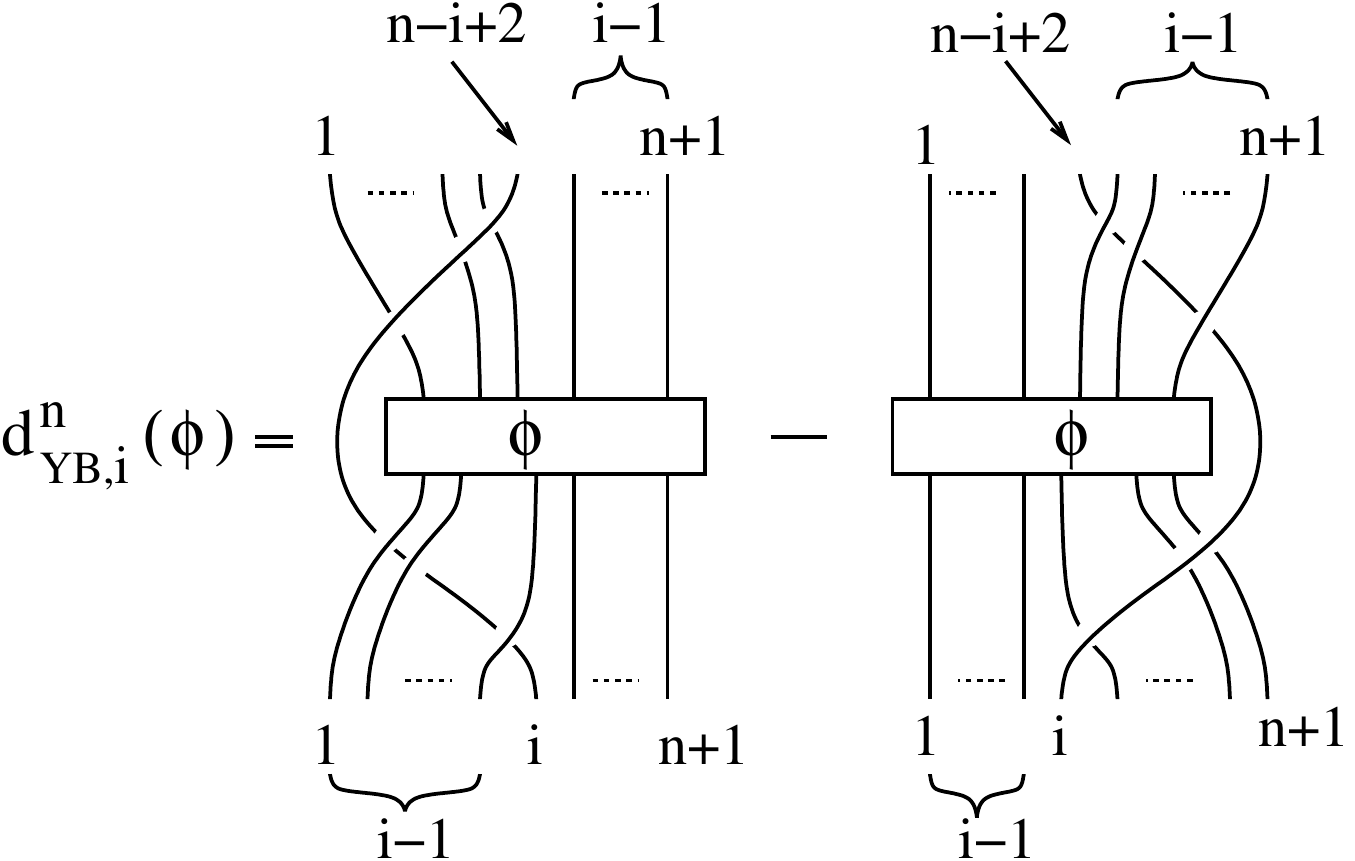}
	\end{center}
	\caption{}
	\label{fig:partial_YB}
\end{figure}

\begin{definition}[\cite{SZ-YBH}] \label{def:YB_coh_def}
		{\rm 
				We define the YB cohomology complex of $V$, with YBO $R$, as the cochain groups $C^n(V,V) := {\rm Hom}_{\mathbb k}(V^{\otimes n},V^{\otimes n})$ and differentials $ \partial^n_{\rm YB} := \sum_{i=1}^n (-1)^id_{{\rm YB},i}^n$ as an alternating sum of partial differentials. The associated cohomology is called {\it Yang-Baxter cohomology} and denoted $H^n_{\rm YB}(V,R)$.
		}
\end{definition}

\begin{lemma}[\cite{SZ-YBH}]  \label{lem:YB_complex}
			The YB cohomology complex of Definition~\ref{def:YB_coh_def} is indeed a cochain complex.
\end{lemma}

		To prove the result it is sufficient  to show that whenever $i<j$ the equality $d_{{\rm YB},j}^{n+1}d_{{\rm YB},i}^n = d_{{\rm YB},i}^{n+1}d_{{\rm YB},j-1}^n$ holds. 
		We include a diagrammatic argument  in Figure~\ref{fig:YBdjdi}  to show this fact. 
		 Specifically, the terms (1) and (5), (2) and (7), (3) and (6), (4) and (8) each cancel pairwise.  
			The first and the last canceling pairs use the YBE for $R$. Also, we notice that the strings coincide in both cases due to the numbering shifts induced by the crossings.

\begin{figure}[htb]
	\begin{center}
		\includegraphics[width=5in]{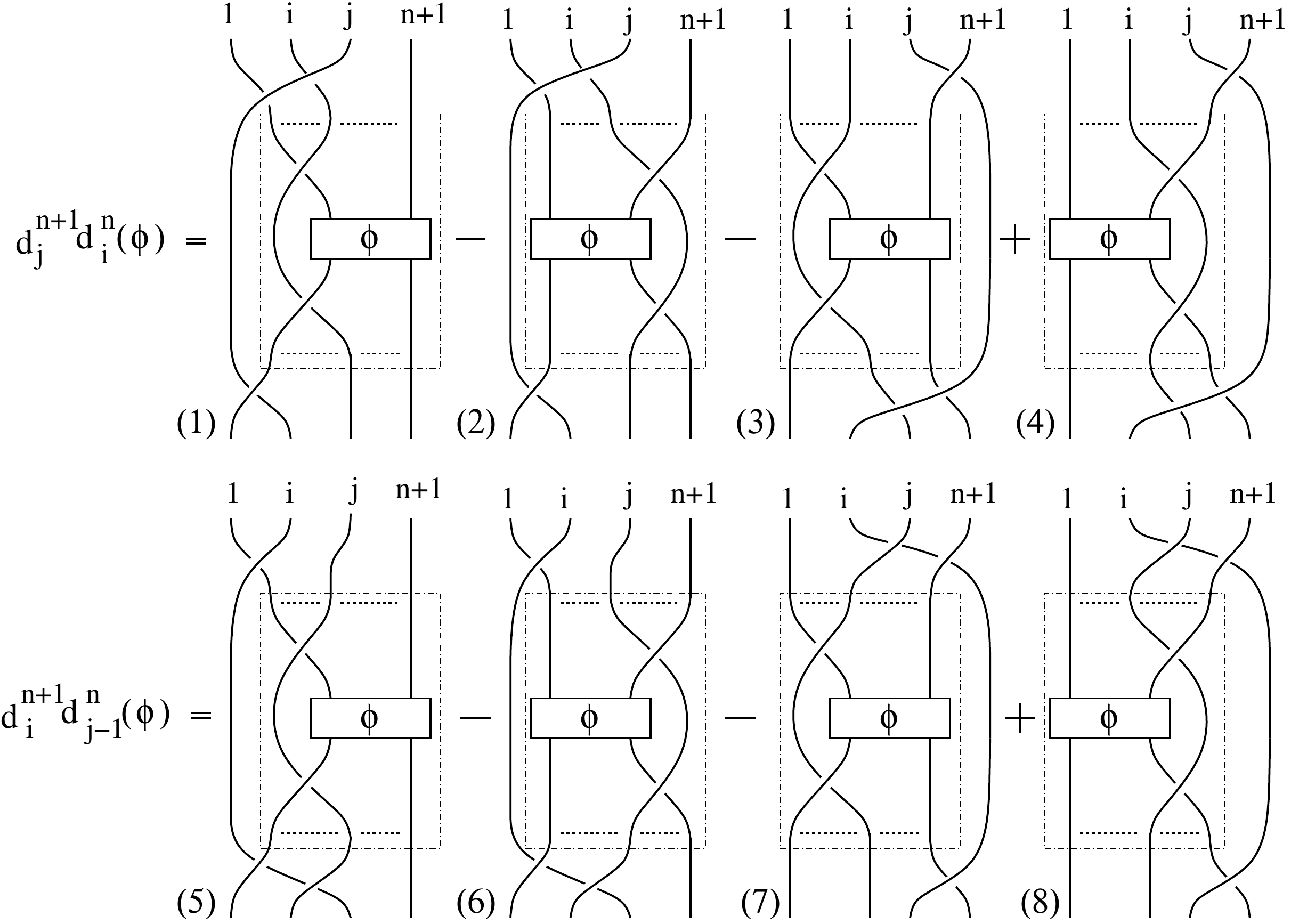}
	\end{center}
	\caption{}
	\label{fig:YBdjdi}
\end{figure}

A natural question to ask at this point is whether the two cohomology theories, from Definition~\ref{def:YB_coh_def} and \cite{Eis}, are related or not. In low degrees, it turns out that when $R$ is invertible (i.e. it is a YBO rather than being just a pre-YBO) the cohomology groups are isomorphic. In the following, we denote by 
$\hat H^n_{\rm YB}(V,R)$ the YB cohomology groups of  \cite{Eis1,Eis}, and  their differentials by $\hat d_{\rm YB}^n$. 
	We recall the following definition, found in \cite{SZ-YBH}. 
		\begin{definition}
			{\rm 
				Let $R : V^{\otimes 2} \longrightarrow V^{\otimes 2}$ be a Yang-Baxter operator over the  $\mathbb k$-module $V$. 
				An operator  $\tilde R: \tilde V^{\otimes 2} \longrightarrow \tilde V^{\otimes 2}$, where  $\tilde V := \mathbb k[[\hbar]]/(\hbar^2)\otimes V$,
				is said to be a {\it Yang-Baxter deformation} of 
				$R$,
				or YB deformation for short, if it satisfies $[R - \tilde R] (V^{\otimes 2}) \subset \hbar V^{\otimes 2}$. 
				Two YB deformations $R_1$ and $R_2$ of $R$ are said to be equivalent if there exists a mapping $f: V \longrightarrow V$ such that 
				\begin{eqnarray*}
						R_2 = R_1 ( f \otimes {\mathbb 1}) + R_1 ( {\mathbb 1} \otimes  f ) 
						-  ( f \otimes {\mathbb 1}) R_1 - ( {\mathbb 1} \otimes  f )  R_1. 
				\end{eqnarray*}
			}
		\end{definition}
		We observe that this definition coincides with the one found in \cite{Eis1}, see Definition~5 therein, upon taking the ideal $\frak m = (\hbar)$. We lastly recall that  it has been shown that the group of YB deformations is isomorphic to the respective second cohomology groups in \cite{Eis,Eis1} and \cite{SZ-YBH}.  See for instance Proposition~27 in \cite{Eis1}, and Theorem~3.9 in \cite{SZ-YBH} where the proof is given for the more general Yang-Baxter-Hochschild cohomology.

\begin{proposition}
		Suppose that $R$ is a YBO. Then, we have
		\begin{eqnarray*}
				H^1_{\rm YB}(V,R) &\cong& \hat H^1_{\rm YB}(V,R)\\
				H^2_{\rm YB}(V,R) &\cong& \hat H^2_{\rm YB}(V,R).
		\end{eqnarray*}
\end{proposition}
\begin{proof}
			Applying $R$ to the differential $\hat d^1_{\rm YB}$ of \cite{Eis}, we find the relation $\partial^1_{\rm YB}(f) = \hat d^1_{\rm YB}(f)R$, for any $1$-cochain $f$. Since $R$ is invertible, we find a 1-1 correspondence between $1$-cocycles of $\partial^1_{\rm YB}$ and $1$-cocycles of $\hat d^1_{\rm YB}$. Since $\partial^0_{\rm YB}$ and $\hat d^0_{\rm YB}$ are the same, this completes the proof of the first isomorphism. 
			To show the second, we just use the fact that both cohomology groups are isomorphic to the group of YB infinitesimal deformations, as shown in \cite{Eis1} and \cite{SZ-YBH}. 
			The groups of deformations of YBOs are the same between the two by taking $\frak m$ in \cite{Eis1} to be $(\hbar)$ in our notation.
			This completes the proof.
\end{proof}

	The relation between $H^n_{\rm YB}(V,R)$ and $\hat H^n_{\rm YB}(V,R)$ remains unclear
	for $n\geq 3$. This does not affect the current paper.

\section{Quantum cocycle invariants  via traces}\label{sec:inf_def}

Let $S = (R, \alpha, \beta, \mu)$ be an EYBO. 
Suppose that there is an infinitesimal deformation 
$$\tilde R=\phi_0 + \hbar \phi_1=R + \hbar \phi
\quad {\rm on} \quad 
\tilde V={\mathbb k} [[ \hbar ]] / (\hbar^2 )  \otimes V . $$
Suppose further  that there is $\tilde \mu = \mu + \hbar \mu_1 : \tilde V \rightarrow \tilde V$ such that
$\tilde S = (\tilde R, \alpha, \beta, \tilde \mu ) $ is an EYBO on $\tilde V$.
Then we say that $\tilde S$ is  an {\it infinitesimal deformation} of an EYBO $S$, or $S$ is deformed to $\tilde S$.  In this situation we also denote the deformation by the pair $(\phi,\mu_1)$, 
and call it an {\it enhanced} Yang-Baxter $2$-cocycle, or EYB $2$-cocycle for short.
We also  simply say that $\phi$ is an EYB $2$-cocycle when $\mu_1$ is understood from context.

\begin{definition}\label{def:deg1_operator}
	{\rm 
			Let $R$ be a YB operator
			with inverse $R^{-1}$, 
			and let $\phi$ be a YB $2$-cocycle. 
			Let $b_m$ be an $m$-braid. Then, we associate an operator 
			$\Psi_\phi(b_m)$  
			 to the braid 
			 as follows. We consider the operator 
			$\tilde R = \phi_0 + \hbar \phi_1$ with $\phi_0=R$ 
			 on 
			 $\tilde V = {\mathbb k} [[ \hbar ]] / (\hbar^2 ) \otimes V$ 
			 	 as above. 
			Observe that $\tilde R$ is invertible with inverse $\tilde R^{-1} = \hat \phi_0 + \hbar \hat \phi_1 $, 
				 where $\hat \phi_0 = R^{-1}$ and $\hat \phi_1 = -R^{-1}\phi_1R^{-1}$. 
				For $R$ and/or $R^{-1}$ we use the notation $R^{\pm}$, and similarly for $\tilde R^{\pm}$.
				 
				 To each standard  generator (and  its inverse) $\sigma^{\pm}_i$, $i=1,\ldots, m-1$, of the braid group $\mathbb  B_m$,
			 we associate the operator $\sigma_i^{\pm}  \mapsto \mathbb 1^{\otimes ( i-1)}\otimes \tilde R^\pm \otimes \mathbb 1^{\otimes (m-i-1)}$. Then, from a presentation of $b_m$ as a product of $\sigma_i^{\pm}$'s we obtain an operator obtained as composition of terms of type $\mathbb 1^{\otimes (i-1)}\otimes \tilde R^\pm \otimes \mathbb 1^{\otimes (m-i-1)}$, which we denote 
			  $\Psi_\phi(b_m)$. 
			 We let $\Psi^i_\phi(b_m)$, $i=0,1$, be defined by 
			 $$
			 \Psi_\phi(b_m) = \Psi^0_\phi(b_m) + \hbar \Psi^1_\phi(b_m) \in 
			 {\rm Hom}(\tilde{V}^{\otimes m},\tilde{V}^{\otimes m}).
			 $$
			
	}
\end{definition}

We use the following notation.

\begin{definition}\label{def:gamma}
	{\rm 
		For given non-negative integers $k$ and $n$, set 
\begin{eqnarray*}
					\Gamma_n^k &=& \{(j_1,\ldots,j_k )\ |\  j_1+\cdots +j_k = n, \  j_\ell \geq 0 \}, \\
					\hat\Gamma_n^k  &=& \{(j_1,\ldots,j_k )\ |\  j_1+\cdots +j_k= n, \  j_\ell \geq 0\ {\rm and}\  j_\ell \neq n,\ {\rm for\ all}\ \ell=1,\ldots,k \}.
\end{eqnarray*}

When $k$ is apparent, we also use $\Gamma_n=\Gamma_n^k$ and $\hat\Gamma_n=\hat\Gamma_n^k$. For $n=1$ we may also omit the subscript. 
}
\end{definition}

\begin{theorem}\label{thm:inf_YB_inv}
Let $S = (R, \alpha, \beta, \mu)$ be an EYBO on $V$. Let $\Gamma=\Gamma^3_1$. 
Then there exists  an infinitesimal deformation $\tilde R=\phi_0 + \hbar \phi_1=R + \hbar \phi$
and $\tilde \mu = \mu + \hbar \mu_1 $ on $\tilde V= {\mathbb k} [[ \hbar ]] / (\hbar^2 )  \otimes V $ such that
$\tilde S = (\tilde R, \alpha, \beta, \tilde \mu ) $ is an EYBO on $\tilde V$ 
if and only if the following equations hold:
	\begin{eqnarray}
			\sum_{(i,j,k)\in \Gamma} (\mu_i\otimes \mu_j)\phi_k &=& \sum_{(i,j,k)\in \Gamma} \phi_i (\mu_j\otimes \mu_k),\label{eqn:inf_mu_comm}\\
			\sum_{(i,j,k)\in \Gamma}\tr_2(\phi_i(\mu_j\otimes \mu_k)) &=& \alpha\beta \mu_1,\label{eqn:inf_partial_tr+}\\
			\sum_{(i,j,k)\in \Gamma}\tr_2(\hat \phi_i(\mu_j\otimes \mu_k)) &=& \alpha^{-1}\beta \mu_1,\label{eqn:inf_partial_tr-}
	\end{eqnarray}
	where $\hat \phi_0 = \phi_0^{-1}$, and $\hat \phi_1 = -\phi_0^{-1}\phi_1\phi_0^{-1}$. 
	
	Moreover, if $\tilde S$ is an EYBO and 
		$b_m$ is an $m$-braid  representing a knot $K$, the quantity 
	\begin{eqnarray}
		\Psi^{1 }_{\tilde S}(b_m) = 			\sum_{(i,j_1,\ldots,j_m)\in \Gamma^{m+1}} \alpha^{-w(b_m)} \beta^{-m}\tr (\Psi^i_\phi(b_m)( \mu_{j_1}\otimes \cdots \otimes \mu_{j_m})) \label{eqn:inv_deg1}
	\end{eqnarray} 
does not depend on  the choice of $b_m$ used, and it is invariant under Markov moves. It is therefore a knot invariant. 
\end{theorem}
\begin{proof}
			We prove the first statement. Observe that given 
 			$ \tilde R = \phi_0 + \hbar \phi_1$, its inverse is given by 
			$\tilde R^{-1} =  \hat \phi_0 + \hbar \hat \phi_1$, up to quadratic terms in $\hbar$, as one might verify directly. 
			The characterization of 
			 the conditions on $\tilde S$ being an EYBO module up to degree 2 gives, quotienting out higher degrees in $\hbar$,
			Equation~\eqref{eqn:mu_comm} and Equation~\eqref{eqn:partial_tr}.
	
			For the second statement, since 
			$\tilde S$ 
			is an EYBO on $\tilde V= \mathbb k[[\hbar]]/(\hbar^2) \otimes  V $,  			using classical results of Turaev in \cite{Tur}, it follows that the trace of $\alpha^{-w(D)} \beta^{-m}\Psi_\phi(b_m)$ is a knot invariant. 
			Since the summation $\sum_{(i,j_1,\ldots,j_m)\in \Gamma^{m+1}}$ collects degree 1 terms with respect to $\hbar$, one sees that it gives precisely the term in \eqref{eqn:inv_deg1}.
\end{proof}

\begin{definition}\label{def:qcocy}
	{\rm 
		From the second statement of this theorem, we denote the Expression~(\ref{eqn:inv_deg1})
	by $\Phi^1_{\tilde S}( K)$, or $\Phi^1_\phi (K)$ 
	when other data for $\tilde S$ are understood, 
	without an explicit reference to the braid $b_m$ used to obtain the operator $\Psi^1_{\tilde S} (b_m)$, as the trace is independent of this choice. 
	 We also denote the Turaev's invariant for the undeformed $S=(R,  \alpha, \beta, \mu )$ 
		by $\Phi^0_\phi( K)$, and combine both by the expression
		$\Phi_\phi(K) = \Phi^0_\phi(K) + \hbar \Phi^1_\phi(K)$. 
				}
\end{definition}

		We now consider the effect of deforming an EYBO by a YB coboundary. 
		
		\begin{lemma}\label{lem:cob_EYBO}
					Let $(R,\alpha,\beta,\mu)$ denote an EYBO, and let $\phi = \delta^1_{\rm YB}(f)$ be a cobounded YB $2$-cocycle. Then, $(\tilde R, \alpha, \beta, \tilde \mu)$ is an EYBO, where $\tilde R = R + \hbar \phi$, and $\tilde \mu = \mu + \hbar \mu_1$, with $\mu_1 = \mu f - f\mu$. 
		\end{lemma}
		\begin{proof}
					We first show \eqref{eqn:inf_mu_comm}, i.e. $\tilde R(\tilde \mu\otimes \tilde \mu) = (\tilde \mu\otimes \tilde \mu)\tilde R$. We need to show that the following equation holds
					\begin{eqnarray*}
							\phi(\mu\otimes \mu) + R(\mu_1\otimes \mu) + R(\mu\otimes \mu_1) &=& (\mu\otimes \mu)\phi + (\mu_1\otimes \mu)R + (\mu\otimes \mu_1)R. 
					\end{eqnarray*}
		For the left-hand side, using
		 the defining formula of  $\phi = \delta^1_{\rm YB}(f)$ and 
		 the fact that $R$ and $(\mu\otimes \mu)$ commute by assumption, we compute
		\begin{eqnarray*}
				\phi(\mu\otimes \mu) + R(\mu_1\otimes \mu) + R(\mu\otimes \mu_1) &=& R(f\mu\otimes \mu) + R(\mu\otimes f\mu) - (f\mu\otimes \mu)R - (\mu\otimes f\mu)R \\
				&& + R(\mu f\otimes \mu)  - R(f\mu \otimes \mu) + R(\mu\otimes \mu f) - R(\mu\otimes f\mu).
		\end{eqnarray*}
		Similarly, for the right-hand side we obtain
		\begin{eqnarray*}
				(\mu\otimes \mu)\phi + (\mu_1\otimes \mu)R + (\mu\otimes \mu_1)R &=& R(\mu f\otimes \mu) + R(\mu\otimes \mu f) - (\mu f\otimes \mu)R - (\mu\otimes \mu f)R\\
				&& + (\mu f\otimes \mu)R - (f\mu \otimes \mu)R + (\mu \otimes \mu f)R - (\mu\otimes f\mu)R. 
		\end{eqnarray*}
		The required equality is therefore seen to hold.

		We now consider \eqref{eqn:inf_partial_tr+} and \eqref{eqn:inf_partial_tr-}, i.e. we need to show the equalities
		\begin{eqnarray*}
				\tr_2(\phi(\mu\otimes \mu)) + \tr_2(R(\mu_1\otimes \mu)) + \tr_2(R(\mu\otimes \mu_1)) &=& \alpha\beta \mu_1 , \\
				\tr_2(\hat\phi(\mu\otimes \mu)) + \tr_2(R^{-1}(\mu_1\otimes \mu)) + \tr_2(R^{-1}(\mu\otimes \mu_1)) &=& \alpha^{-1}\beta \mu_1,
		\end{eqnarray*}
		where $\hat\phi := -R^{-1}\phi R^{-1}$ is the infinitesimal component of the inverse of $\tilde R$. Observe that
		\begin{eqnarray*}
				R^{-1}\phi R^{-1} 
				&=& R^{-1}\delta^1_{\rm YB}(f) R^{-1}   \\
				&=& R^{-1}R(f\otimes \mathbb 1)R^{-1} + R^{-1}R(\mathbb 1\otimes f)R^{-1} - R^{-1}(f\otimes \mathbb 1)RR^{-1} + R^{-1}R(\mathbb 1\otimes f)RR^{-1}\\
				&=& (f\otimes \mathbb 1)R^{-1} + (\mathbb 1\otimes f)R^{-1} - R^{-1}(f\otimes \mathbb 1)+ R^{-1}R(\mathbb 1\otimes f)\\
				&=& - \delta^1_{  R^{-1} } (f),  
		\end{eqnarray*}
		where $\delta^1_{ R^{-1}}(f)$ is the first 
		 YB differential of the operator $R^{-1}$ applied to $f$. 

		Therefore,  if $(R,\alpha, \beta,\mu)$ is an EYBO, then  we have that $(R^{-1},\alpha^{-1},\beta,\mu)$ is also an EYBO. 
		Hence it is sufficient to prove \eqref{eqn:inf_partial_tr+}, 
		since \eqref{eqn:inf_partial_tr-} is equivalent to \eqref{eqn:inf_partial_tr+} for the EYBO $(R^{-1},\alpha^{-1},\beta,\mu)$, with  $\hat \phi =  \hbar \delta^1_{ R^{-1}}(f)$ 
		and $\tilde \mu= \mu +\hbar\mu_1$.		
		For the left-hand side of \eqref{eqn:inf_partial_tr+} we have 
		\begin{eqnarray*}
				\lefteqn{\tr_2(\phi(\mu\otimes \mu)) + \tr_2(R(\mu_1\otimes \mu)) + \tr_2(R(\mu\otimes \mu_1))}\\ 
				&=& \tr_2[R(f\otimes \mathbb 1)(\mu\otimes \mu) + R(\mathbb 1\otimes f)(\mu\otimes \mu)
				 - (f\otimes \mathbb 1)R(\mu\otimes \mu) - (\mathbb 1\otimes f)R(\mu\otimes \mu)\\
				 && \hspace{1cm} + R(\mu f\otimes \mu) - R(f\mu \otimes \mu) + R(\mu\otimes \mu f) - R(\mu\otimes f\mu)]\\
				 &=& \tr_2[R(f\otimes \mathbb 1)(\mu\otimes \mu) + R(\mathbb 1\otimes f)(\mu\otimes \mu)
				 - (f\otimes \mathbb 1)(\mu\otimes \mu)R - (\mathbb 1\otimes f)(\mu\otimes \mu)R\\
				 && \hspace{1cm} + R(\mu f\otimes \mu) - R(f\mu \otimes \mu) + R(\mu\otimes \mu f) - R(\mu\otimes f\mu)]\\
				 &=& -\tr_2[(f\otimes \mathbb 1)(\mu\otimes \mu)R] 
				 + \tr_2[R(\mu\otimes \mu)(f\otimes \mathbb 1)]\\
				 && - \tr_2[(\mathbb 1\otimes f)(\mu\otimes \mu)R] 
				 + \tr_2[R(\mu\otimes \mu)(\mathbb 1\otimes f)]\\
				 &=& -f\tr_2[(\mu\otimes \mu)R] 
					 + \tr_2[R(\mu\otimes \mu)]f - \tr_2[(\mathbb 1\otimes f)(\mu\otimes \mu)R] 
			+ \tr_2[R(\mu\otimes \mu)(\mathbb 1\otimes f)]\\
				&=& -f\alpha\beta \mu
				+ \alpha\beta\mu f\\
				 &=& \alpha\beta[\mu f-f\mu]\\
				 &=& \alpha\beta\mu_1,   
		\end{eqnarray*}
		where we have used \eqref{eqn:mu_comm} and \eqref{eqn:partial_tr} for $R$ and $\mu$, 
		along with 
		\begin{eqnarray*}
				- \tr_2[(\mathbb 1\otimes f)(\mu\otimes \mu)R] + \tr_2[R(\mu\otimes \mu)(\mathbb 1\otimes f)]=0,
		\end{eqnarray*}
		 by the cyclic property of the partial trace with respect to maps that factor as $\mathbb 1\otimes f$, and 
		 \begin{eqnarray*}
		 		\tr_2[(f\otimes 1)(\mu\otimes \mu)R] &=& f\tr_2((\mu\otimes \mu)R), \\
		 		\tr_2[R(\mu\otimes \mu)(f\otimes 1)] &=& \tr_2[R(\mu\otimes \mu)]f.
		 \end{eqnarray*}
		 This completes the proof. 
	\end{proof}

	The following is an analogue, in our setting, of the result for quandle invariants that states that the invariant of a coboundary is the number of quandle colorings, see \cite{CJKLS}.

	 \begin{proposition} \label{prop:cob}
	 			Let $K$ be a knot. Let $(R,\alpha,\beta,\mu)$ be an EYBO, let  $\phi := \delta^1_{\rm YB}(f)$ be a coboundary for some $1$-cochain $f$, and set $\mu_1 = \mu f - f\mu$. Then, we have $\Phi_\phi(K) = \Phi_\phi^0(K)$. In other words, the infinitesimal part of the quantum cocycle invariant vanishes, $\Phi_\phi^1(K)=0$.  
	 \end{proposition}
 	\begin{proof}
 			Applying Lemma~\ref{lem:cob_EYBO}, it follows that $\Phi_\phi(K)$ is well defined.
 			Let $b_m$ denote a braid on $m$-string whose closure is $K$. We 
			show that 
			$$\alpha^{w(b_m)}\beta^m\Phi_\phi^1(K) = \sum_{(i,j_1,\ldots, j_m)\in \Gamma^{m+1}} \tr(\Psi_\phi^i(b_m)(\mu_{j_1}\otimes \cdots \otimes \mu_{j_m})) =0. $$

			 In this proof, we will refer to the arcs of $b_m$ 
			 that are bounded by crossings (regardless of whether the arc in question is an over or under arc at the end crossings), and {\it external} otherwise. External arcs have at least one end at the top or bottom of the braid, and  internal arcs do not.			 
			We claim that the only terms in the operator $\Psi_\phi^1(b_m)$ that do not vanish are those where $\phi$ appears in a position such that it has either external incoming arcs, or external outgoing arcs. In fact, if an arc is internal, then it takes part in two consecutive crossings. 
			The term $\Psi_\phi^1$ consists of the braid $b_m$ where each crossing is replaced by $R^\pm$, except for one of the crossings where $\phi$ appears. Moreover, $\phi$ appears exactly once in each position. If we consider an internal arc, then there are two terms in the sum defining $\Psi_\phi^1$ where $\phi$ appears, once in the top crossing and once in the bottom crossing. 
			 There are three different types of configurations associated to this situation. Namely, we can have $b_m = b'_m\sigma_i^\pm\sigma_i^\pm b''_m$, or $b_m = b'_m \sigma_i^\pm\sigma_{i+1}^\pm b''_m$, or $b_m = b'_m\sigma_{i+1}^\pm\sigma^\pm_ib''_m$.

	  \begin{figure}[htb]
	\begin{center}
		\includegraphics[width=3.8in]{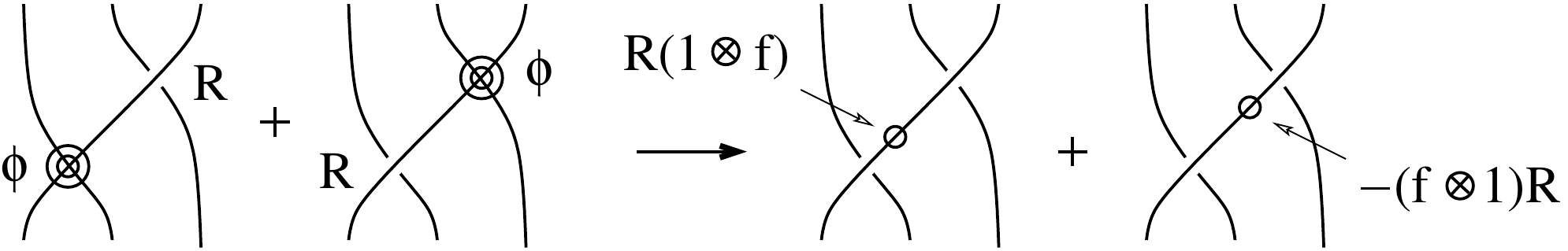}
	\end{center}
	\caption{}
	\label{cancel}
\end{figure}

			 First, we consider the case where all crossings are positive.
			 We consider only the case $b_m = b'_m\sigma_i\sigma_{i+1}b''_m$, as the other two are substantially the same, and we show the fact that the contribution to $\Psi^1_\phi(b_m)$ vanishes for the terms where $\phi$ appears in the position of $\sigma_i$, or $\sigma_{i+1}$. 
 			These terms correspond to the sum 
 			\begin{eqnarray*}
 					\Psi_\phi^0(b'_m)\circ (\phi\otimes \mathbb 1)(\mathbb 1\otimes R)\circ  \Psi_\phi^0(b''_m) + \Psi_\phi^0(b'_m)\circ (R\otimes \mathbb 1)(\mathbb 1\otimes \phi)\circ  \Psi_\phi^0(b''_m),
 			\end{eqnarray*}
 			which is part of the sum defining $\Psi_\phi^1(b_m)$.
			In the left-hand side of Figure~\ref{cancel}, these two terms are depicted. 
				Since 
			$$\phi = \delta^2_{\rm YB}(f) = R(f\otimes \mathbb 1) + R(\mathbb 1\otimes f) - (f\otimes \mathbb 1)R - (\mathbb 1\otimes f)R, $$ 
			 we can pair the elements  corresponding to $R(f\otimes \mathbb 1)$ associated to 
			 $\Psi_\phi^0(b'_m)\circ (\phi\otimes \mathbb 1)(\mathbb 1\otimes R)\circ  \Psi_\phi^0(b''_m), $
			  and  $- (\mathbb 1\otimes f)R$ in 
			  $\Psi_\phi^0(b'_m)\circ (R\otimes \mathbb 1)(\mathbb 1\otimes \phi)\circ  \Psi_\phi^0(b''_m). $
			  These terms are seen to cancel. 
			  In the right-hand side of Figure~\ref{cancel} these 1-cochains are depicted as circles on the corresponding edges. The two terms represented by circles on the middle edge
			  are the canceling pair. 
			 The cancelations of 1-cocycles on arcs are also seen in diagrams of Figure~\ref{YBdiff1}, where circles from a single YB 2-cocycle are distributed to each adjacent edge, and two circles coming from crossings that bound each internal arc represent canceling terms.
             This diagrammatic interpretation was also used in \cite{CJKLS} for quandle cocycle invariants.
			
			 Since, as shown in Lemma~\ref{lem:cob_EYBO}, $\phi^{-1} = \delta^2_{R^{-1}}(f)$, it follows that the computation is substantially the same when $\sigma_i$ is replaced by $\sigma_i^{-1}$, upon changing $R$ with $R^{-1}$. So, the same calculation as above can be applied to the cases where a negative crossing appears, and we see that the terms all cancel.
			  We 
			  proceed in the same manner for each internal arc, therefore pairing all the terms arising from each instance of $\phi$ appearing in the positions of the crossings of $b_m$, and cancel all the terms where $f$ appears on internal arcs. The only terms that are potentially nonzero in $\Psi^1_\phi(b_m)$ are therefore those in which $f$ appears on external arcs.
	
						By adding canceling pairs $\sigma_i \sigma_i^{-1}$ if necessary, we may assume that all strings cross at least another string of the braid. 
			Further, observe that all instances where $f$ appears as a top arc have positive sign, while all instances where $f$ appears on a bottom arc have negative sign. Therefore, we have
 			\begin{eqnarray*}
 					\lefteqn{\sum_{(i,j_1,\ldots, j_m)\in \Gamma^{m+1}} \tr(\Psi_\phi(b_m)(\mu_{j_1}\otimes \cdots \otimes \mu_{j_m}))}\\
 					&=& \tr(\Psi^1_\phi(b_m)(\mu_0\otimes \cdots \otimes \mu_0)) + \sum_{(j_1,\ldots, j_m)\in \Gamma^m} \tr(\Psi^0_\phi(b_m)(\mu_{j_1}\otimes \cdots \otimes \mu_{j_m}))\\
 					&=& \sum_{i=1}^m \tr[\Psi^0_\phi(b_m)(\mu_0^{\otimes (i-1)  }\otimes f\mu_0\otimes \mu_0^{\otimes (m-i-1) }) - (\mathbb 1^{\otimes (i-1) }\otimes f\otimes \mathbb 1^{\otimes ( m-i-1) })\Psi^0_\phi(b_m)(\mu_0 \otimes \cdots \otimes \mu_0)]\\
 					&& + \sum_{i=1}^m \tr(\Psi^0_\phi(b_m)(\mu_0^{\otimes (i-1) }\otimes \mu_1\otimes \mu_0^{\otimes (m-i-1) })) 
					\end{eqnarray*}
					\begin{eqnarray*}
 					&\stackrel{(A)}{=} & \sum_{i=1}^m \tr[\Psi^0_\phi(b_m)(\mu_0^{\otimes (i-1)}\otimes f\mu_0\otimes \mu_0^{\otimes (m-i-1)}) - \Psi^0_\phi(b_m)(\mu_0^{\otimes (i-1) }\otimes \mu_0f\otimes \mu_0^{\otimes (m-i-1) })]\\
 					&& + \sum_{i=1}^m \tr(\Psi^0_\phi(b_m)(\mu_0^{\otimes (i-1)}\otimes \mu_1\otimes \mu_0^{\otimes (m-i-1)  }))\\
 					&=& - \sum_{i=1}^m \tr(\Psi^0_\phi(b_m)(\mu_0^{\otimes (i-1)}\otimes \mu_1\otimes \mu_0^{\otimes (m-i-1) })) + \sum_{i=1}^m \tr(\Psi^0_\phi(b_m)(\mu_0^{\otimes (i-1)}\otimes \mu_1\otimes \mu_0^{\otimes (m-i-1)}))\\
 					&=& 0
 			\end{eqnarray*}
 where the equality $(A)$ follows from the cyclic property of the partial trace.
 \end{proof}

\section{Examples   via traces  }\label{sec:traces}

In this section we provide examples that yield nontrivial values of the quantum cocycle invariant defined in the preceding section via traces.

\subsection{Invariants from deformations of the transposition}
			Let us consider a finite dimensional vector space $V$ along with the transposition map $R=\tau : V\otimes V \rightarrow V\otimes V$, which is a YBO on $V$. 
						
\begin{lemma} 		For $R=\tau$, we have 	
$H^2_{\rm YB}(V, V) = C^2_{\rm YB}(V, V) = {\rm Hom} (V^{\otimes 2}, V^{\otimes 2})$.
\end{lemma}

\begin{proof}			For a $2$-cochain 
			$\phi$, 
			 we use the structure constants $\phi(e_i\otimes e_j) = \phi_{ij}^{k\ell}e_k\otimes e_\ell$, where we employ the usual Einstein notation for the sums of repeated indices, and we have denoted by $e_1, \ldots, e_n$ a basis for $V$. 
			The first term in the 2-cocycle condition is 
			\begin{eqnarray*}
				(R \otimes \mathbb 1) (\mathbb 1 \otimes R)(\phi \otimes \mathbb 1)(e_i \otimes e_j \otimes e_k)
				= (\tau  \otimes \mathbb 1) (\mathbb 1 \otimes \tau) ( \phi_{ij}^{tr} e_t \otimes e_r \otimes e_k) ) = 
				\phi_{ij}^{tr}  e_k \otimes e_t \otimes e_r.
			\end{eqnarray*}
					 Similarly the 2-cocycle condition is computed as 
			\begin{eqnarray*}
		\lefteqn{	\delta^2_{\rm YB}(\phi)(e_i \otimes e_j \otimes e_k ) }\\
		&=& 			\phi_{ij}^{tr} e_k\otimes e_t\otimes e_r + \phi_{ik}^{tr} e_t\otimes e_j\otimes e_r 
					+ \phi_{jk}^{tr} e_t\otimes e_r\otimes e_i \\ 
					& & - \phi_{jk}^{tr} e_t\otimes e_r\otimes e_i -  \phi_{ik}^{tr} e_t\otimes e_j\otimes e_r  - \phi_{ij}^{tr} e_k\otimes e_t\otimes e_r =0 , 
			\end{eqnarray*} 
			which is 
			satisfied. It follows that any $2$-cochain is a $2$-cocycle. Similarly, using the notation $f(e_i) = f_i^je_j$ for a $1$-cochain $f:V\rightarrow V$, we find directly that $\delta^1_{\rm YB}(f) = 0$. 
			Hence we have $H^2_{\rm YB}(V, V) = C^2_{\rm YB}(V, V)$. 
\end{proof}

For the transposition $\tau$, Equations~\eqref{eqn:mu_comm} and~\eqref{eqn:partial_tr}  can be checked 
			to see that $S=(R,\alpha, \beta, \mu_0)= (\tau, 1, 1, \mathbb 1)$ is an EYBO.
		We consider the problem of finding enhanced $2$-cocycles of $\tau$.

\begin{lemma} \label{lem:cocy_tau}
Let $V = \mathbb C\langle e_0,e_1\rangle$ be a 2-dimensional real
vector space on basis $e_0, e_1$. 
Let $S=(\tau, 1, 1, \mathbb 1)$ be an EYBO as above. 
Let $\tilde S=(\tilde R, 1, 1, \tilde \mu)$ where $\tilde R = \tau + 	\hbar \phi$ and $\tilde \mu= \mathbb 1 + \hbar \mu_1$.
Consider the following choices for $\phi$ and $\mu_1$:
\begin{eqnarray} \quad
\phi_{i0}^{k0} = - \phi_{0i}^{0k}, \ 
\phi_{i1}^{k1} = - \phi_{1i}^{1k}, \
\phi_{01}^{10} \ {\rm and} \  \phi_{10}^{01} \ {\rm arbitrary,\  with} \ 
 \mu_1 \ {\rm defined \ by}\  \mu_i^k = -\phi_{i0}^{k0} - \phi_{i1}^{k1}.
\label{eq:phi}
\end{eqnarray}
Then 
we have that 
$\tilde S =(\tau + \hbar \phi, 1, 1, \mathbb 1 + \hbar \mu_1)$ is an EYBO. 

\end{lemma}

\begin{proof}
	 Let $\tilde \mu = \mu_0 + \hbar \mu_1 $ be a deformation of $\mu_0$,
			 where $\mu_0=\mathbb 1$, and write $\mu_1(e_i) = \mu_i^je_j$. 
			We need  to obtain the constraints on $\phi$ and $\mu_1$ in order to have Equation~\eqref{eqn:inf_mu_comm} hold. 
			The first term in the left-hand side of Equation~\eqref{eqn:inf_mu_comm}   
			for $(0,0,1) \in \Gamma_2$ is 
		$\phi_0( \mu_0 \otimes \mu_1)(e_i \otimes e_j )$, and substituting $\phi_0=\tau$, $\mu_0=\mathbb 1$
		and $\mu_1 (e_j ) = \mu_j^k e_k $, we obtain $\mu_j^k e_k\otimes e_i$ as the first term. 
					Similar computations for  Equation~\eqref{eqn:inf_mu_comm}  give  
			\begin{eqnarray*}
					\mu_j^k e_k\otimes e_i + \mu_i^k e_j\otimes e_k + \phi_{ij}^{k\ell}e_k\otimes e_{\ell} = \phi_{ij}^{k\ell}e_k\otimes e_{\ell} + \mu_i^ke_j\otimes e_k + \mu_j^ke_k\otimes e_i,
			\end{eqnarray*}
			which once again holds identically. 
			
			For the deformation
			$\tilde S=(\tilde R, \alpha, \beta, \phi)=(\tau+\hbar \phi, \mathbb 1, \mathbb 1, \mu_1)$ to give 
			an  EYBO, it is sufficient to verify Equations~\eqref{eqn:inf_partial_tr+} and \eqref{eqn:inf_partial_tr-}.
			To evaluate $\tr_2$, we calculate as follows.
			With basis $\{ e_i \}$,  let an operator $A$ defined via the tensor notation   
			$A(e_i\otimes e_j) = \sum_{kl} a_{ij}^{kl}e_k\otimes e_l$,  
			and let $e_j^*$ be  dual basis vectors. 
			Then the  computation for $\tr_2(A)$ is performed as  
			$$ e_i \mapsto \sum_j e_i\otimes e_j\otimes e_j^* 
			\mapsto \sum_j A(e_i\otimes e_j)\otimes e_j^* 
			\mapsto \sum_j a_{ij}^{kl} e_k\otimes e_l \otimes e^*_j 
			\mapsto \sum_j a_{ij}^{kl}\delta_{jl} e_k 
			=  \sum_j a_{ij}^{kj} e_k.  $$
			Hence the entry $ik$ of the matrix for $\tr_2(A)$ is $ \sum_j a_{ij}^{kj} $, 
		    or simply $a_{ij}^{kj}$ in Einstein convention.

			We compute 
			$\tr_2 \tau(\mu_0\otimes \mu_1)(e_i) =  \mu_i^ke_k = \mu_1(e_i)$, 
			$\tr_2 \tau(\mu_1\otimes \mu_0)(e_i) =   \mu_i^je_j = \mu_1(e_i)$, 
and 
			 $\hat \phi_1 = -\phi_0^{-1}\phi_1\phi_0^{-1}$  as in Theorem~\ref{thm:inf_YB_inv}. 
			For Equations~\eqref{eqn:inf_partial_tr+} we compute
			$$\tr_2 \tau(\mu_0\otimes \mu_1)(e_i) + \tr_2 \tau(\mu_1\otimes \mu_0)(e_i) + \tr_2 \phi(\mu_0\otimes \mu_0)(e_i) =\mu_1 (e_i), $$
			 hence we obtain $ \phi_{ij}^{kj} = - \mu_i^k$.
			Similarly for  Equations~\eqref{eqn:inf_partial_tr-} we obtain $ \phi_{ji}^{jk} = \mu_i^k$.
			Then it is checked that the 
			 stated Solution~\eqref{eq:phi} 
			indeed satisfies these equations, due to the symmetry $\phi_{ik}^{i\ell} = - \phi_{ki}^{\ell i}$, which uniquely fixes $\mu_1$, and the fact that $\phi_{01}^{10}$ and $\phi_{10}^{01}$ do not appear in the equations for the partial trace.
\end{proof}

		Let $\sigma_1$ denote the standard generator of the 2-strand braid group ${\mathbb B}_2$,
and let $T_n$ be the closure of $\sigma_1^n$ for s positive integer $n$.
For $n=2,3$, $T_n$ is a Hopf link and trefoil knot, respectively.

\begin{proposition}\label{prop:Tn}
Consider the special case of  Solution~\eqref{eq:phi} in Lemma~\ref{lem:cocy_tau}:
 $$\phi_{01}^{01} = -\phi_{10}^{10} = \phi_{01}^{11} = - \phi_{10}^{11} = 1, \
 \phi_{01}^{10} = \phi_{10}^{01} = q, $$
  and $\phi_{ij}^{k\ell} = 0$ otherwise, where $q$ is an arbitrary real number, along with $\mu_0^0 = \mu_0^1 = -\mu_1^1 = 1$, $\mu_1^0 = 0$.

For the knots and links $T_n$ and the YB 2-cocycle $\phi$ in 
this solution, 
we have the following values. 
\begin{eqnarray*}
				\Phi_{\phi} (T_n) =\begin{cases}
					4+2nq\hbar \ \ n = 2k  \\
					2  \ \ n = 2k+1
				\end{cases}
\end{eqnarray*}
\end{proposition}

\begin{proof}
		We consider $\Phi_\phi^i(T_n)$ for $i=0,1$ separately. For $i=0$, $\Phi_\phi^i(T_n)$ is the trace of $\tau^n$, which takes values $4$  or  
		$2$ depending on the parity of $n$. For the degree $i=1$ we first compute the traces 
		$$\tr (\phi) = \tr (\tau\phi\tau) =0, \quad {\rm and} \quad \tr (\phi\tau) = \tr (\tau\phi) = 2q, $$
		which are obtained by direct inspection of the action of the operators on the four basis vectors $e_i\otimes e_j$. 
	Moreover, a direct computation shows that 
	$$\tr (\tau[\mu_1\otimes \mathbb 1 + \mathbb 1\otimes \mu_1]) = \tr (\mu_1\otimes \mathbb 1 + \mathbb 1\otimes \mu_1) = 0, $$
	 which in turn implies that $\tr (\tau^n[\mu_1\otimes \mathbb 1 + \mathbb 1\otimes \mu_1])=0$ for all $n$. 
     Therefore we do not include $\mu_1$ terms in the computations below. 
			
		We now proceed by induction. 
		We need to compute 
		\begin{eqnarray*}
				\Phi^1_\phi(T_n) &=& \sum_{i=0}^{n-1} \tr (\tau^{n-i-1}\phi\tau^i), 
		\end{eqnarray*} 
			and proceed by induction on $k$, where we distinguish the two cases $n=2k$ and $n=2k+1$, with $k\geq 0$. The base case $k=0$ of $n=2k+1$ is clearly satisfied, since $\Phi^1_\phi(T_n) = \tr (\phi) = 0$. Suppose that $\Phi^1_\phi(T_n) = 0$ for $n = 2k-1$. We compute
		\begin{eqnarray*}
					\Phi^1_\phi(T_{2k+1}) 
					&=&  \sum_{i=0}^{2k} \tr (\tau^{2k-i}\phi\tau^i)\\
					&=& \sum_{i=0}^{2k-2} \tr (\tau^2\tau^{2k-i}\phi\tau^i) + \tr (\tau \phi \tau^{2k-1}) + \tr (\tau^2 \phi \tau^{2k})\\
					&=& \sum_{i=0}^{2k-2} \tr (\tau^{2k-i}\phi\tau^i) + \tr (\tau \phi \tau) + \tr (\phi)\\
					&=& 0. 
		\end{eqnarray*}
		For the even case, $n = 2k$, we have the base case $\Phi^1_\phi(T_2) =   \tr (\phi\tau) +  \tr (\tau\phi) = 4q$. Suppose now the result holds for $n=2k$. We 
		show it for $n=2k+2$. As in the previous case we have 
		\begin{eqnarray*}
			\Phi^1_\phi(T_{2k+2}) 
			&=&  \sum_{i=0}^{2k+1} \tr (\tau^{2k-i+1}\phi\tau^i)\\
			&=& \sum_{i=0}^{2k} \tr (\tau^2\tau^{2k-i+1}\phi\tau^i) + \tr (\tau \phi \tau^{2k}) + \tr (\tau^2 \phi \tau^{2k+1})\\
			&=& \sum_{i=0}^{2k} \tr (\tau^{2k-i+1}\phi\tau^i) + \tr (\tau \phi) + \tr (\phi\tau)\\
			&=& 4kq+4q.  
		\end{eqnarray*}
		This completes the proof. 
\end{proof}

Since the invariant of the unknot and unlink are  $2$ and $4$, we find that the quantum cocycle invariant is nontrivial for this choice of enhanced $2$-cocycle. 

\subsection{Invariants from the tensor version of quandle 2-cocycle invariants} \label{subsec:Q}

We now discuss a class of example of importance, which motivates the name of quantum cocycle invariants.
	
	A {\it quandle} $(Q, *)$ is a set $Q$ with an idempotent operation  $*$ such that 
	the right translation $R_y: x \mapsto x*y$, $x\in Q$,  is a $*$-automorphism for all $y$. 
	A group with the conjugation operation is a typical example of a quandle.
	A quandle homology theory was developed in \cite{CJKLS} and quandle cocycle invariant was defined with 
	applications to knots and knotted surfaces. 
	
	The quandle 2-cocycle invariant for a knot diagram is defined as follows. Let  $Q$ be a finite quandle and $A$ be a finite abelian group.
			A quandle 2-cocycle with the coefficient group $A$ is a function $\psi: Q \times Q \rightarrow  A$ with an
			equation called quandle 2-cocycle condition
			 $$ \psi (x,y) + \psi (x*y,z) - \psi(x,z) - \psi (x * z, y * z ) =0 $$
			for all $x,y,z \in Q$. 
						
Let $D$ be an oriented  diagram of an oriented  knot $K$.
A {\it coloring} of $D$ by $Q$ is a map ${\cal C} : {\cal A} \rightarrow Q$
from the set ${\cal A}$ of arcs of $D$ (broken at under crossings) to $Q$ such that at each crossing 
the coloring rule is satisfied as depicted in  Figure~\ref{coloring}, where the left is a positive crossing and the right is negative.
Let ${\rm Col}_Q(D)$ be the set of colorings of $D$ by $Q$. 
Let $\psi: Q \times Q \rightarrow A$ be a quandle 2-cocycle with a coefficient abelian group $A$ with multiplicative notation. 
For a given coloring ${\cal C}$,  the pair of colors $(x,y)$ 
as indicated in Figure~\ref{coloring} at a crossing, is denoted for each crossing $u$ by $(x_u, y_u)$.
A weight $\psi (x_u, y_u)^{\epsilon (u) }$ is assigned at each crossing $u$, where $\epsilon(u)$ is a sign $\pm 1$ of a crossing, for a  positive or negative crossing $u$, respectively.
Then the {\it quandle 2-cocycle invariant} is defined by 
$$\Phi_\psi^{\rm Q} (D)=  \sum_{{\cal C} \in {\rm Col}_Q(D)} \prod_{u} \psi (x_u, y_u)^{\epsilon (u) } , $$
where the colors $(x_u, y_u)$ at each crossing depends on a given coloring ${\mathcal C}$, and 
the summation is taken over all colorings.
This is known to be independent of choice of a diagram $D$ of a knot $K$,
where the quandle 2-cocycle condition is used to show invariance under Reidemeister move III,  and is denoted by $\Phi_\psi^{\rm Q} (K)$.
We refer to \cite{CJKLS} for details.

\begin{figure}[htb]
\begin{center}
\includegraphics[width=2.5in]{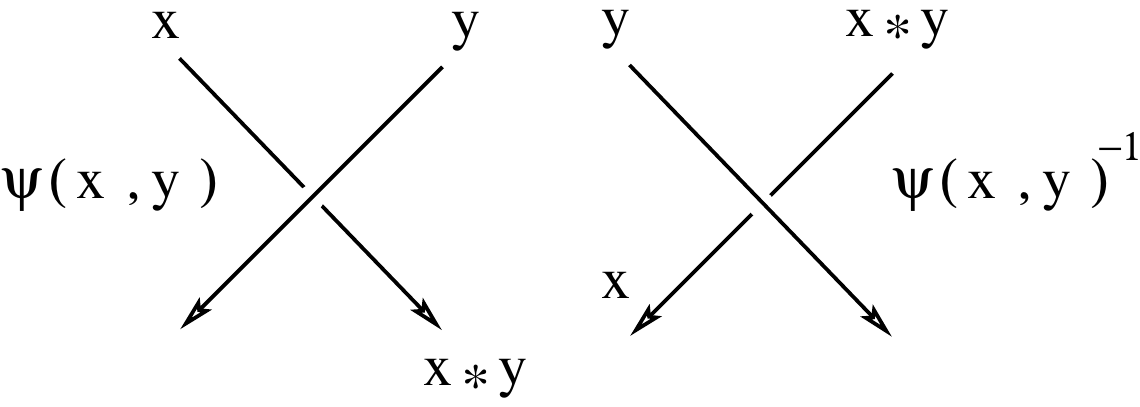}
\end{center}
\caption{}
\label{coloring}
\end{figure}

			Let $(Q,*)$ be a quandle, and let $\mathbb k$ denote a ring. Let $V$ be the free module generated by $Q$ over $\mathbb k$. In this example, the tensor product refers to $\otimes_{\mathbb k}$.
			 We define the YBO on $V$ as $R(x\otimes y) = y\otimes  (x*y)$, and extending it by linearity. 
						 We now show that $R$ along with the choice of $\alpha = \beta =1$ and $\mu_0 = \mathbb 1$ gives an EYBO
			 $S=(R, \alpha, \beta, \mu_0)=(R, 1, 1, \mathbb 1 )$.

\begin{lemma}		\label{lem:Q2toYB2}
	Set  $\phi(x\otimes y) := \psi (x, y)\cdot R(x\otimes y)$.
Then  $\phi$ is a YB 2-cocycle.
	\end{lemma}

\begin{proof}		The first term of the YB 2-differential $ \delta^2_{\rm YB} (\phi) (x \otimes y \otimes z) $
			for basis elements is 
\begin{eqnarray*}
\lefteqn{  (R \otimes {\mathbb 1} ) ( {\mathbb 1}  \otimes R ) ( \phi \otimes  {\mathbb 1} ) (x \otimes y \otimes z) }\\
&=&  (R \otimes {\mathbb 1} ) ( {\mathbb 1}  \otimes R ) (\psi (x , y) [  y \otimes (x*y) ] \otimes z)  \\
&=&   (R \otimes {\mathbb 1} ) (  \psi (x , y) ( y \otimes z)  \otimes [ ( x*y)*z ] ) \\
&=&  \psi (x , y ) (z \otimes ( y * z ) \otimes [ ( x*y)*z ] ) . 
\end{eqnarray*} 
This computation is similar to the one for quandle 2-cocycles, and it is seen that 
$\phi$ satisfies the YB 2-cocycle condition, from the quandle 2-cocycle condition of $\psi$. 
\end{proof}
	
Lemmas~\ref{lem:YB2toR} and \ref{lem:Q2toYB2} imply the following.

	\begin{lemma}\label{lem:YBcocyR}
	Set $\mathbb k = \Z [A]$, where $A$ is a coefficient abelian group of a quandle 2-cocycle $\psi$
	of a quandle $Q$, 
	and $V$ the free $\mathbb k$-module on which $R$ is defined.
	Define the operator 
	$$\tilde R(x\otimes y) = R(x\otimes y) + \hbar \phi(x \otimes y) \quad 
	{\rm on } \quad \tilde V = \mathbb k [[\hbar] / (\hbar^2) \otimes V . $$
	Then $\tilde R$ is a YBO on $\tilde V$.
\end{lemma}

\begin{lemma}\label{lem:Rinverse}
Let $Q$ be a quandle, and for  $x, y \in Q$, 
let $z \in Q$ be  the unique element such that $z*x=y$. 
Set   $\hat \phi (x \otimes y) := - \psi (z, x ) R^{-1} (x \otimes y) $.
Then $R^{-1} + \hbar \hat \phi$ is the inverse of $\tilde R$.  
\end{lemma}

\begin{proof}
We note that $R (z \otimes x)=(x \otimes y)$ and $R^{-1} (x \otimes y)=(z \otimes x)$.
Then one computes 
$$ (R + \hbar \phi ) (R^{-1} + \hbar \hat \phi) ( x\otimes y) 
= (R + \hbar \phi )[  (z \otimes x) + \hbar (- \psi (z , x) (z \otimes x) ]
= x \otimes y .
$$
The other case is similar.
\end{proof}

\begin{lemma}		\label{lem:QEYBO}
Let $\tilde R= R+ \hbar \phi $
as in Lemma~\ref{lem:YB2toR}.	Define  $\tilde S=( \tilde R, 1, 1, \tilde \mu)$,
			where $\tilde \mu = \mu_0 + \hbar \mu_1$ with $\mu_0=\mathbb 1$ and $\mu_1=0$.
			Then $\tilde S$ is an EYBO.
\end{lemma}			

\begin{proof}			 			This follows from $\psi(x\otimes x) = 0$.
	Since Lemma~\ref{lem:YBcocyR} shows that $\tilde R$ 
	  is a YBO, we need only to show that the condition on partial traces \eqref{eqn:partial_tr} is satisfied. For notational clarity, 
	  denote by $e_x$ the elements of $V$ corresponding to $Q$, instead of using the letter $x$ directly. We denote by $e_x^*$ the dual element, defined by $\langle e_x| e_y^*\rangle := e_y^*(e_x) = \delta_{xy}$. Then, for each basis vector $e_x$, to compute the partial trace ${\rm Tr}_2(\tilde R\circ\tilde \mu\otimes \tilde \mu)$ we have
	 	\begin{eqnarray*}
	 			e_x
	 			&\mapsto& \sum_{y\in Q} e_x \otimes e_y \otimes e_y^*\\
	 			&\mapsto& \sum_{y\in Q} [e_y \otimes e_{x*y} \otimes e_y^*+\hbar \phi(e_x\otimes e_y) \otimes e_y^*]\\
	 			&=& \sum_{y\in Q} [e_y \otimes e_{x*y} \otimes e_y^*+\hbar \psi(x,y)\cdot  e_y \otimes e_{x*y} \otimes e_y^*]\\
	 			&\mapsto& \sum_{y\in Q}[ \langle e_{x*y}|e_y^*\rangle e_y + \hbar \psi(x,y) \langle e_{x*y}|e_y^*\rangle e_y]\\
	 			&=& e_x + \psi(x,x)\cdot e_x\\
	 			&=& e_x, 
	 	\end{eqnarray*}
 		where we have used the fact that $x*y=y$ in a quandle implies $x=y$. A similar computation holds for ${\rm Tr}_2(\tilde R^{-1}\circ\tilde \mu\otimes \tilde \mu)$. 
\end{proof}

Lemma~\ref{lem:QEYBO} and Theorem~\ref{thm:inf_YB_inv} imply the following.

\begin{proposition} \label{prop:QtoYB}
The EYBO $\tilde S$ defined
from $\tilde R=R + \hbar \phi$  in Lemma~\ref{lem:QEYBO}
			 defines  via trace a YB 2-cocycle invariant 
			$\Phi_\phi (K) = \Phi_\phi^0 (K) + \hbar \Phi_\phi^1 (K)$ for knots $K$. 
			\end{proposition}

We now show a relation between quandle 2-cocycle invariants and the YB 2-cocycle invariants of knots. Then we have the following result.

\begin{proposition}\label{pro:quandle_quantum}
We have 
$ \Phi^1_\phi (K) = \sum_{\mathcal C  \in {\rm Col}_Q(D)} \sum_u \epsilon(u) \psi(x_u, y_u) $,
where both summations are  taken in $\mathbb k=\Z [A]$. Thus we have
$$ \Phi_\phi (K) = \Phi_\phi^0 (K) + \hbar \Phi_\phi^1 (K)
= | {\rm Col}_Q(D) | + \hbar  \sum_{\mathcal C  \in {\rm Col}_Q(D)} \sum_u \epsilon(u) \psi(x_u, y_u) , $$
where $ | {\rm Col}_Q(D) | $ denotes the number of colorings of $D$ by $Q$.
\end{proposition}

\begin{proof}
Let ${\cal C}$ be a coloring of a diagram $D$ of a knot $K$ by a quandle $Q$. 
 Let $D$ be a diagram in a closed braid form of an $n$-braid $b_m$ of $K$. 
As in the proof of Lemma~\ref{lem:QEYBO}, we use a basis of $V$ denoted by $e_x$, where $x\in Q$, dual basis
vectors  $e_x^*$ and the pairing 
$$\langle e_{x_1}\otimes \cdots \otimes e_{x_k}|e^*_{y_k} \otimes \cdots \otimes e^*_{y_1} \rangle = \prod_i \delta_{x_i y_i}. $$

At a crossing with a coloring pair $(x,y)$ as in Figure~\ref{coloring}, the map 
$$\tilde R (x \otimes y) = R (x \otimes y) + \hbar \psi (x \otimes y)  R (x \otimes y)$$
 is assigned, and the composition is taken 
over all crossings, as in Definition~\ref{def:deg1_operator}.
 For the inverse, 
 $$\tilde R^{-1} (x \otimes y) = R^{-1}  (x \otimes y) - \hbar \psi(y*x, y) R^{-1}   (x \otimes y) $$
  is assigned by Lemma~\ref{lem:Rinverse}. 
More specifically, suppose that $K$ has a braid presentation as the closure of 
 $b_m = \sigma_{i_1}^{\epsilon(i_1)}\cdots \sigma_{i_k}^{\epsilon(i_k)}$. 
Recall that we make  the assignments $\sigma_i^{\pm}  \mapsto \mathbb 1^{\otimes ( i-1 )}\otimes \tilde R^\pm \otimes \mathbb 1^{\otimes (m-i-1)}$. 
We consider the terms of $\hbar$-degree at most one, so that
we take the terms where at most  one of the braid generators receive the assignment 
$\sigma_i^{\pm}  \mapsto \mathbb 1^{\otimes ( i-1)}\otimes \tilde R^\pm \otimes \mathbb 1^{\otimes (m-i-1)}$
and others $\sigma_i^{\pm}  \mapsto \mathbb 1^{\otimes (i-1)}\otimes  R^\pm \otimes \mathbb 1^{\otimes (m-i-1)}$.
	Then, $\Phi^1_\phi(K)$ 
	is written as 
	\begin{eqnarray*}
			\Phi^1_\phi(K) 
			&=& {\rm Tr}(\Psi^1_\phi(b_m))\\
			&=& 
						\sum_{x_1, \ldots, x_m\in Q} \langle \Psi^1_\phi(b_m)(e_{x_1}\otimes \cdots \otimes e_{x_m})|e^*_{x_m}\otimes \cdots \otimes e^*_{x_1}\rangle\\
			&=& 
			\sum_{x_1, \ldots, x_m \in Q}\sum_{r=1}^k \epsilon(i_r)
			 \langle \sigma_{i_k}^{\epsilon(i_k)} \cdots  \mathbb (1^{\otimes (i_r-1) }\otimes \phi \otimes \mathbb 1^{\otimes (m-i_r-1) })  \cdots \sigma_{i_1}^{\epsilon(i_1)}(e_{x_1}\otimes \cdots \otimes e_{x_m})\\
			&&\hspace{3cm} |e^*_{x_m}\otimes \cdots \otimes e^*_{x_1}\rangle\\
	&=&
	\sum_{x_1, \ldots, x_m\in Q}\sum_{r=1}^k \epsilon(i_r)\psi(y^{(r)}_{i_r},y^{(r)}_{i_{r+1}}) \langle e_{y^{(k)}_1}\otimes \cdots \otimes e_{y^{(k)}_m}|e^*_{x_m}\otimes \cdots \otimes e^*_{x_1}\rangle\\
			&=& 
			 \sum_{\mathcal C\in {\rm Col}_Q(D)} \sum_{r=1}^k \epsilon(i_r)\psi(y^{(r)}_{i_r},y^{(r)}_{i_{r+1}}), 
	\end{eqnarray*}
	where we have indicated by $y^{(r)}_h$ 
	 the coloring of the $h^{\rm th}$ string after applying the first $r$ 
	 operators in the braid presentation of $b_m$, so that $y^{(k)}_{i_r}$ 
	 is the the color of the output of $\Psi_\phi(b_m)$. 
			The last expression represents  $\sum_{\mathcal C  \in {\rm Col}_Q(D)}  \sum_{x \in {\rm Cv[\mathcal C]} } x \pmod{\hbar^2} $, where ${\rm Cv[\mathcal C]}$ denotes the coloring vectors associated to a coloring $\mathcal C$. 
	Since  $\Phi_\phi^0 (K) = | {\rm Col}_Q(D) | $ without cocycle contributions in the above computations, we obtain the formula in the statement.
			\end{proof}

\begin{remark} 
{\rm 
Recall that the quandle cocycle invariant is defined by  
$$\Phi_\psi^{\rm Q} (D)=  \sum_{{\cal C} \in {\rm Col}_Q(D)} \prod_{u} \psi (x_u, y_u)^{\epsilon (u) } . $$
The difference to $\Phi^1_\phi$ is that in the quandle invariant the Boltzmann weights associated to each coloring $\mathcal C$ are multiplied together in $A$, and then the corresponding elements in $\mathbb Z[A]$ for each coloring are summed together using the additive operation of the group ring; however, in the quantum cocycle invariant the Boltzmann weights for each coloring are summed in $\mathbb Z[A]$ using the additive operation of the group ring. If for each fixed coloring we multiply the elements of $\mathbb Z[A]$, we obtain the corresponding summand of the quandle cocycle invariant. 

In the example that follows, the quantum cocycle  invariant turns out to be stronger than the quandle cocycle  invariant.
It is desirable to know more explicit and detailed relations between the two. 
For example, the following questions arise: 
 Are there examples of knots with the same quantum invariant but distinct quandle cocycle invariant?
 If the quantum version is always stronger, can we characterize the possible values  of quantum cocycle invariants with a given quandle cocycle invariant? 
 }
\end{remark}

	\begin{figure}[htb]
		\begin{center}
			\includegraphics[width=2.2in]{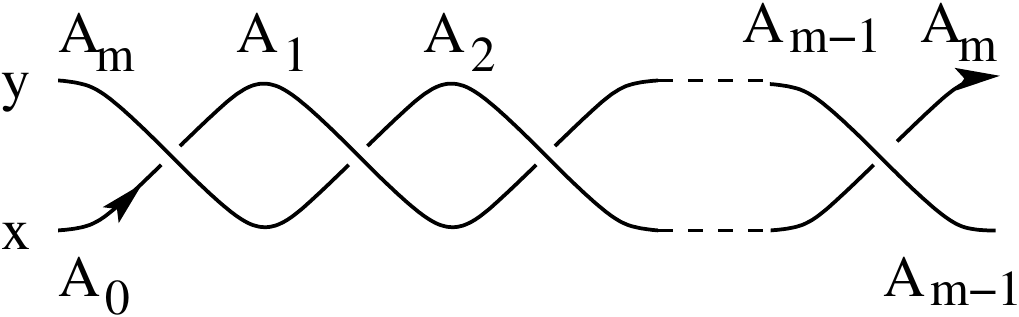}
		\end{center}
		\caption{}
		\label{Tm}
	\end{figure}

	\begin{example}\label{ex:quandleinv}
		{\rm 
			We compare the values of quandle cocycle invariant and the corresponding quantum cocycle invariant values
			for $(2,m)$-torus knots $T_m$ 
			with the Alexander quandle $Q=\Z_2[t]/(t^2 + t + 1)$, with operation defined as $a*b := ta + (1-t)b$.
			We represent elements of $Q$ by $\{0,1,t, 1+t\}$.  Recall that each element of $Q$ gives rise to a constant coloring which we will call a {\it trivial coloring}. 
			 In \cite{CJKS}, it was shown that the following gives a non-trivial 2-cocycle  with the coefficient group $\Z_2$: 
			$$\psi(x,y)=\sum_{(a,b)}  \chi_{(a,b)}(x,y),$$
			 where $\chi_{(a,b)}$ is the characteristic function
			$$
					\chi_{(a,b)}(x,y) =\begin{cases}
						1 \ \ (x,y) = (a,b)\\
						0 \ \ {\rm otherwise}
					\end{cases}, 
			$$ 
			and the sum ranges over 
			all $a,b \in Q$ such that $a\neq b$ and $a,b \in \{ 0,1,1+t\}$. 
			It was shown by computations in  \cite{CJKS} that this particular choice of values of $a,b$  missing $t$ indeed gives a 2-cocycle. 
			Let $T_m$ be a torus knot (or a link) with $m$ crossings that is the closure of the diagrams  depicted in Figure~\ref{Tm}. The arcs are labeled by $A_i$ for $i=0, \ldots, m$ in the figure, and let $u_j$ be the crossings from left to right of the figure.
			Although the cocycle invariant is known for $T_m$ with this $Q$, we present computations here to observe the contrast of computations to the quantum cocycle  invariant.
			
			Let $x, y$ be a pair of colors assigned to the left arcs labeled by $A_0$ and $A_m$
			as indicated in the figure.
			Then one computes that the colors  are successively
			$$ \mathcal C (A_1)= tx + (1+t)y, \quad  \mathcal C (A_2)=x, \quad \mathcal C (A_3)= y, \quad 
			\mathcal C (A_4)= tx + (1+t)y, $$
			and this pattern repeats with period $3$. Hence $T_m$ is 
			nontrivially  colored  if and only if $m$ is a multiple of 3, $m=3\ell$. 
			If $m$ is not divisible by 3, then the quandle cocycle invariant is $\Phi^{\rm Q}_\psi (T_m)=4$.

			A coloring of $T_m$  for $m=3 \ell$ is trivial if and only if $x=y$. 
			We first examine the cocycle invariant for $T_3$ (trefoil).
			For distinct $(x,y)$, set $z=tx + (1+t)y$, which is the color of $A_2$. 
			The cocycle values for $u_i$ for $i=1,2,3$ are $\psi(x,y)$, $\psi(y,z)$ and $\psi(z,x)$, respectively, 
			and this pattern is repeated for more crossings. 
			
			By direct computations for 12 pairs $(x,y)$, $x \neq y$, for every ordered triple $(x,y,z)$, the number of $t \in Q$ 
			among $(x,y,z)$ is either 0 or 1. 
			If the number of $t$ among $\{ x,y,z \}$ is 0, then 
			from the definition $\psi(x,y)=\sum_{(a,b)}  \chi_{(a,b)}$ summed over $a,b \in \{ 0,1,1+t\}$, all of 
			$\psi(x,y)$, $\psi(y,z)$, and  $\psi(z,x)$ take the value 1, hence 
			$$\psi(x,y)+\psi(y,z)+ \psi(z,x) = 1 \in \Z_2  $$
			in additive notation.
			If the number of $t$ is 1, then exactly one of $\psi(x,y)$, $\psi(y,z)$, and  $\psi(z,x)$ is 1 and the other two take the value  0. Hence again we have $\psi(x,y)+\psi(y,z)+ \psi(z,x) = 1 \in \Z_2  $. 
			
			With a multiplicative generator $\zeta $ of $\Z_2$, $\psi(x,y)\psi(y,z) \psi(z,x) =\zeta$,
			so that every nontrivial coloring contributes $\zeta$ to $\prod_{u} \psi(x_u, y_u)^{\epsilon (u) } $.
			Hence  we obtain the cocycle invariant value  $\Phi^{\rm Q}_\psi (T_3)= 4 + 12 \zeta$. Since the colorings and cocycle values repeat, we obtain 
			$$\Phi^{\rm Q}_\psi (T_{3 \ell} ) = 4 + 12 \zeta$$
			for all positive integer $\ell$. 
					
			Next we compute the corresponding quantum cocycle invariant.
			Let $\phi(x \otimes y)=\psi (x \otimes  y) R(x \otimes y)$ as in Lemma~\ref{lem:Q2toYB2}.
			In the above analysis of colorings, among the 12 pairs $(x,y)$, there are two cases: 
			\begin{eqnarray*}
				{\rm (Case \ 1)} & & 
				\psi(x,y)=\psi(y,z)= \psi(z,x) =1 , \\
				{\rm (Case \ 2)} & &  \{ \psi(x,y), \psi(y,z),  \psi(z,x) \} = \{ 0,0,1\} .
			\end{eqnarray*}
			where the numbers of colorings are 3 and 9, respectively.
			In (Case 1), the contribution to $\Phi^1_\phi (T_3)$ is $3 \hbar$, and in (Case 2), it is $\hbar$. 
			Hence in total the contribution
			to $\sum_u \epsilon(u) \psi(x_u, y_u)$  is $3 (3 \hbar) + 9 (\hbar) = 18 \hbar$.
			The trivial colorings give the contribution $4$.
			Hence by Proposition~\ref{pro:quandle_quantum} we obtain 
			$\Phi_\phi (T_3)= 4 + 18 \hbar$.
			
			For general $m=3 \ell$, the contributions increase according to the number of crossings, and 
			(Case 1) gives contribution $3 \ell \hbar$ and (Case 2) gives $\ell \hbar$.
			In total we obtain $3 (3\ell  \hbar) + 9 (\ell \hbar) = 18 \ell \hbar$ and 
			we obtain 
			$$\Phi_\phi (T_{3 \ell})= 4 + 18 \ell  \hbar . $$
					We note that for this particular example the quantum cocycle version gives a stronger invariant
			than the quandle cocycle invariant. Further studies on relations between these invariants are desirable.
			
		} 
	\end{example}

\section{Quantum cocycle invariants  via maxima/minima 
 } \label{sec:cupcap}

In this section we describe a construction of quantum cocycle invariants by deforming maxima and minima
for knot diagrams with a height function specified.
 Throughout this section, a height function is fixed on the plane where knot diagrams lie,
and diagrams are in general position, so that they have isolated maxima, minima, and double crossings.
Diagrams are unoriented.

Let $R$ be a YBO on $V$.
Let $\tilde R = \phi_0 + \hbar \phi_1$, where $\phi_0 := R$ and $\phi_1 := \phi$, be a deformed YBO. 
Denote by $\tilde R^{-1} = \hat \phi_0 + \hbar \hat \phi_1$ as before.
Let 
$\tilde \cup = \cup_0 + \hbar \cup_1$, where $\cup_0 := \cup$ and $\cup_1: V \otimes V \rightarrow V$ is a pairing.
Similarly let  $\tilde \cap = \cap_0 + \hbar \cap_1$, where $\cup_0 := \cap$ and $\cap_1:  V  \rightarrow V \otimes V $ is a copairing.
Suppose that  $\cup$ and $\cap$ satisfy the switchback, passcup, and passcap properties.
Below we use the notation $\Gamma=\Gamma_1^2$ from Definition~\ref{def:gamma}.

\begin{theorem} \label{thm:cupcap}
The deformed YBO $\tilde R$,  pairing  $\tilde \cup$ and copairing $\tilde \cap$ on $\tilde V =  \mathbb k[[\hbar]]/(\hbar^2) \otimes V $
satisfy the switchback, passcup, and passcap properties if 
and only if the following equations hold.
\begin{eqnarray}
\sum_{(i,j) \in \Gamma} (\cup_i \otimes {\mathbb 1}) ( {\mathbb 1} \otimes \cap_j) 
&=&  {\mathbb 1} \label{switch1} \\
\sum_{(i,j) \in \Gamma} ( {\mathbb 1} \otimes \cup_i)  (\cap_j \otimes {\mathbb 1}) 
&=&  {\mathbb 1} \label{switch2} \\
\sum_{(i,j) \in \Gamma} ( {\mathbb 1} \otimes \cup_i) (\phi_j \otimes  {\mathbb 1} )
&=& \sum_{(i,j) \in \Gamma} (\cup_i \otimes {\mathbb 1}) ( {\mathbb 1} \otimes \hat \phi_j ) \label{passcup1} \\
\sum_{(i,j) \in \Gamma} ( {\mathbb 1} \otimes \cup_i) ( \hat \phi_j \otimes  {\mathbb 1} )
&=& \sum_{(i,j) \in \Gamma} (\cup_i \otimes {\mathbb 1}) ( {\mathbb 1} \otimes  \phi_j ) \label{passcup2} \\
\sum_{(i,j) \in \Gamma} (\phi_i \otimes {\mathbb 1} ) ( {\mathbb 1} \otimes \cap_j) 
&=& \sum_{(i,j) \in \Gamma}( {\mathbb 1} \otimes \hat \phi_i) (\cap_j \otimes {\mathbb 1} )  \label{passcap1} \\
\sum_{(i,j) \in \Gamma} (\hat \phi_i \otimes {\mathbb 1} ) ( {\mathbb 1} \otimes \cap_j) 
&=& \sum_{(i,j) \in \Gamma}( {\mathbb 1} \otimes  \phi_i) (\cap_j \otimes {\mathbb 1} )  \label{passcap2}
\end{eqnarray}

Moreover, if 
$\tilde \phi$, $\tilde \cup$ and $\tilde \cap$ 
satisfy the switchback, passcup and passcap identities, then they 
define  a knot invariant 
 $$\Phi^{\rm M}_{(\tilde R, \tilde \cup, \tilde \cap)}(K) = \Phi^{{\rm M} 0}_{(\tilde R, \tilde \cup, \tilde \cap)}(K) + \Phi^{{\rm M} 1}_{(\tilde R, \tilde \cup, \tilde \cap)}(K)$$ up to regular isotopy
(equivalence without type I Reidemeister moves).

If in addition 
$$ \sum_{(i,j) \in \Gamma}\cup_j  \phi_i =  \cup_1, \quad  \sum_{(i,j) \in \Gamma} \phi_j \cap_i =  \cap_1 $$
are satisfied, it provides a knot invariant up to isotopy.
\end{theorem}

\begin{proof}
The YBO $\hat \phi$ 
on $\tilde V \otimes \tilde V$, where $\tilde V = V\otimes \mathbb k[[\hbar]]//(\hbar^2)$, 
together with $\tilde \cup$ and $\tilde \cap$ satisfy the switchback, passcup, and passcap properties 
if they satisfy the corresponding conditions on infinitesimal part (degree one part with respect top $\hbar$),
which are computed  in Equations (\ref{switch1}) through (\ref{passcap2}).
Then the composition corresponding to a given knot diagram with a height function defines a knot invariant
up to regular isotopy with these operators.
For isotopy, the condition for type I move translates to the stated condition on the degree one term of $\hbar$.
\end{proof}

\begin{figure}[htb]
\begin{center}
\includegraphics[width=3.5in]{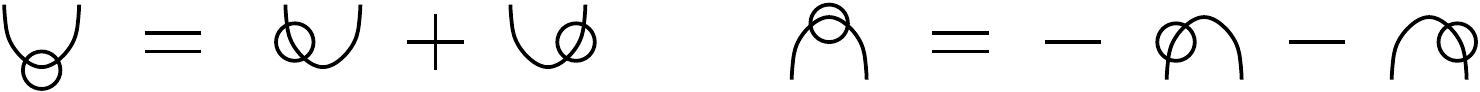}
\end{center}
\caption{}
\label{swd1}
\end{figure}

For the rest of the section, we show that the cobounded cocycles give trivial infinitesimal part, in parallel to
Lemma~\ref{lem:cob_EYBO} and Proposition~\ref{prop:cob}.

\begin{lemma}\label{lem:cupcapcob}
Let $\phi=\delta^1(f)$  be a cobounded YB 2-cocycle
and $\tilde R=R + \hat \phi$ be the deformed YBO on $\tilde V$.
Let $\tilde \cup = \cup + \hbar \bar \cup$ and $\tilde \cap = \cap  + \hbar \bar \cap$ be a deformed 
pairing and copairing on $\tilde V$, respectively, where
$$\bar \cup := \cup (f \otimes \mathbb 1 + \mathbb 1 \otimes f ) \quad 
{\rm and } \quad \bar \cap :=  (- f \otimes \mathbb 1 - \mathbb 1 \otimes f) \cap $$
as represented by diagrams in Figure~\ref{swd1} left and right.
Then $(\tilde R, \tilde \cup, \tilde \cap)$ satisfy the conditions in Theorem~\ref{thm:cupcap}.
\end{lemma}

\begin{proof}
The switchback property Equation~\eqref{switch2} is checked directly as diagrammatically represented in Figure~\ref{swd2}. 
Equation~\eqref{switch1} is similarly checked by their mirror diagrams.

A direct computation as in the proof of Lemma~\ref{lem:cob_EYBO} shows that $\tilde R^{-1} = R^{-1} + \hbar \hat \phi$, where $\hat \phi = \delta^1_{R^{-1}}(f)$ is the image of the first YB coboundary of the inverse YBO $R^{-1}$. A direct computation for the terms of degree $1$ in $\hbar$ for the passcap and passcup moves using $\tilde R$ and $\tilde R^{-1}$ shows that the equations hold. 
A part of a passcap move is depicted in Figure~\ref{cancelcap}. 
In the figure, one of the terms $R(\mathbb 1 \otimes f)$ of the coboundary from $\phi$ is depicted 
as a circle in the middle edge of the right-hand side, and another term $-(f \otimes \mathbb 1 )\cap $ from
$\bar \cap$ is depicted as another circle on the same edge, and seen to cancel.
The other terms involving $f$ are assigned on the end edges.
The passcap equality of $R$ and $\cap$ imply the corresponding  equality for $\tilde R$ and $\tilde \cap$.
\end{proof}

\begin{figure}[htb]
\begin{center}
\includegraphics[width=3.5in]{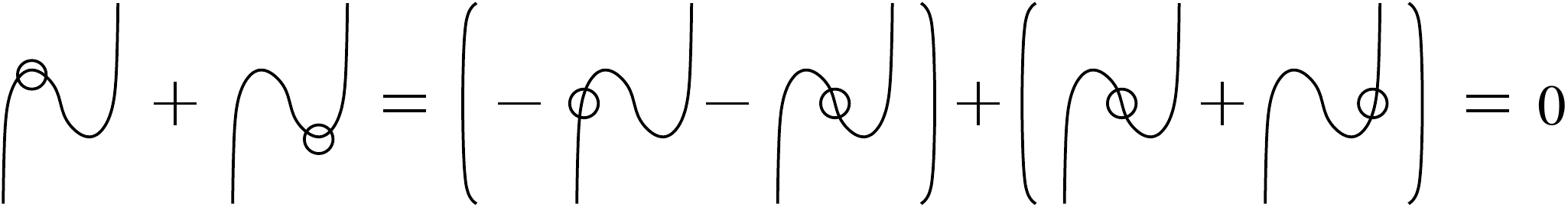}
\end{center}
\caption{}
\label{swd2}
\end{figure}

By  Theorem~\ref{thm:cupcap}, the deformed 
YBO and (co)pairings $(\tilde R, \tilde \cup, \tilde \cap)$ of Lemma~\ref{lem:cupcapcob}
defines a regular isotopy invariant 
$\Phi_{(\tilde R, \tilde \cup, \tilde \cap)}^{\rm M} =\Phi_{(\tilde R, \tilde \cup, \tilde \cap)}^{\rm M 0} + \hbar \Phi_{(\tilde R, \tilde \cup, \tilde \cap)}^{\rm M1}  $.

\begin{figure}[htb]
\begin{center}
\includegraphics[width=3.5in]{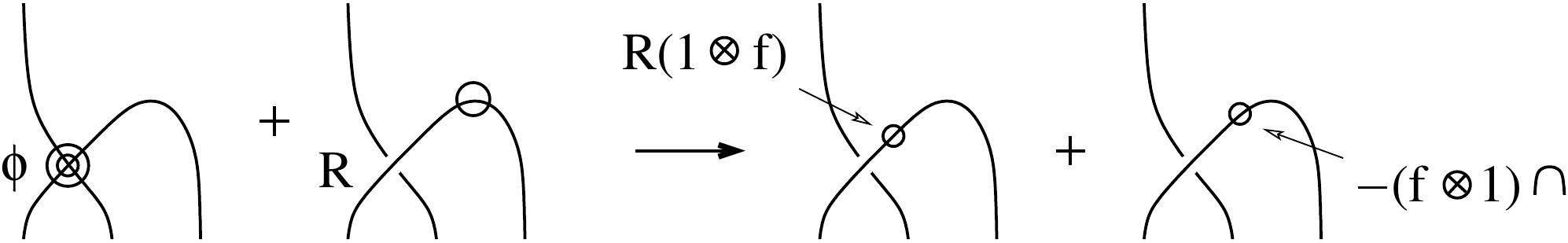}
\end{center}
\caption{}
\label{cancelcap}
\end{figure}

\begin{proposition}\label{prop:cupcapcob}
The infinitesimal part of the  invariant $\Phi_{(\tilde R, \tilde \cup, \tilde \cap)}^{\rm M1}  $
is trivial for a cobounded YB 2-cocycle $\phi=\delta^1(f)$. 
In other words, it holds that $\Phi_{(\tilde R, \tilde \cup, \tilde \cap)}^{\rm M} =\Phi_{(\tilde R, \tilde \cup, \tilde \cap)}^{\rm M 0}$.
\end{proposition}

\begin{proof}
This is proved in a manner similar to Proposition~\ref{prop:cob}. 
There are three types of arcs in a given diagram in this case: (1) internal edge bounded by crossings,
(2) edges with one end point on a crossing, other on an extremum, and (3) both ends on extrema. 
For the internal edge (1), cancelation happens as in Figure~\ref{cancel}.
Edges of case (2) happens where diagrams are as in passcup or passcap moves, and cancelation happens as in Figure~\ref{cancelcap}. Case (3) happens as in figures of the switchback move, and cancelations happen as in Figure~\ref{swd2}. Hence the same argument as in Proposition~\ref{prop:cob} applies. 
\end{proof}

\section{Examples via maxima/minima } \label{sec:bracket}

The Kauffman bracket polynomial \cite{K&P}
for knots up to regular isotopy is defined by the skein relation 
in Figure~\ref{skeinbracket}.
 The bracket value for the unknot diagram with no crossing is set to be $1$.
In this section we examine  invariants defined from YB $2$-cocycles and the bracket following the definition in Section~\ref{sec:cupcap}.
In this section we assume that $\mathbb k$ contains  $\sqrt{-1}$  and an invertible element $A$
that serves as a variable in the Kauffman bracket.

\begin{figure}[htb]
\begin{center}skeinbracket.pdf
\includegraphics[width=2.4in]{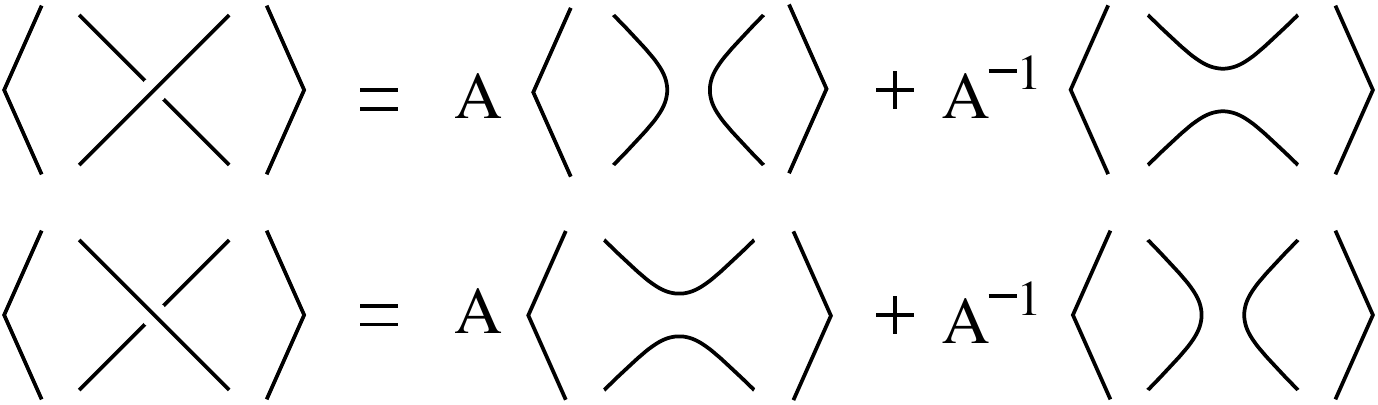}
\end{center}
\caption{}
\label{skeinbracket}
\end{figure}

Let $\cup$ and $\cap$ be pairing used in the bracket skein relation defined in \cite{K&P}
for basis elements $e_1$, $e_2$ of $V={\mathbb k}^{\otimes 2}$ by $\cup(e_i\otimes e_i)=0$ for $i=1,2$ and
$$\cup(e_1 \otimes e_2)=\sqrt{-1} A , \quad \cup(e_2 \otimes e_1)=- \sqrt{-1} A^{-1} . $$
The copairing is defined by 
$$\cap (1)=\sqrt{-1} A e_1 \otimes e_2 - \sqrt{-1} A^{-1} e_2 \otimes e_1. $$
 It is known that these satisfy the switchback, passcup, and passcap identities~\cite{K&P}. 
The pairing is represented by a cup (a local minimum of a diagram), and 
the copairing is represented by a cap (a local maximum of a diagram) in Figure~\ref{skeinbracket}.
As a first approach we examine the case where the pairing and copairing are not deformed,
so that $\bar \cup = 0 = \bar \cap$.

\begin{figure}[htb]
\begin{center}
\includegraphics[width=3.7in]{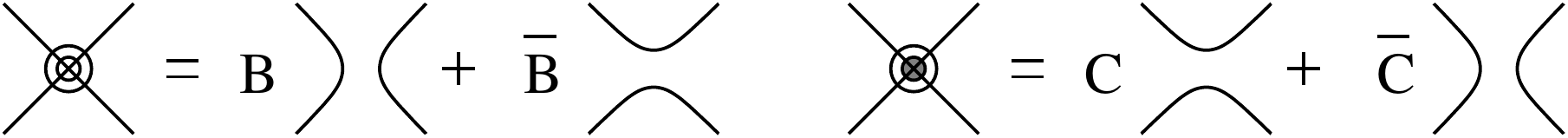}
\end{center}
\caption{}
\label{cocyskein}
\end{figure}

 Diagrammatic convention for a $2$-cocycles for $\phi_1$ and $\hat \phi_1$  are depicted in  Figure~\ref{cocyskein}, where a double circle represents $\phi$ and a  double circle  
 with shaded inner circle represents $\hat \phi$.
 Let $B, \bar B, C, \bar C \in \mathbb k$. 
 We set $\phi= B \mathbb 1 + \bar B \cap \cup$ and 
 $\hat \phi= C  \cap \cup + \bar C \mathbb 1 $ as in Figure~\ref{cocyskein1}
  and examine this case.
  Thus with the above $\phi$, $\hat \phi $ and undeformed $\cup$ and $\cap$, our set up is
  $ \tilde R = R + \hbar \phi$ and $ \tilde R^{-1} = R^{-1} + \hbar \hat \phi$ modulo $ \hbar^2$ as described in the preceding section.
 
  In Figure~\ref{cocyskein1} and later figures, the square bracket indicates that the cocycle skein relations are to be applied locally, and to be replaced by the terms in the right-hand side, where the Kauffman bracket is applied. This local skein relation is applied to all crossings successively. For simplicity, we drop brackets when the stages of computations are understood.

\begin{figure}[htb]
\begin{center}
\includegraphics[width=4.2in]{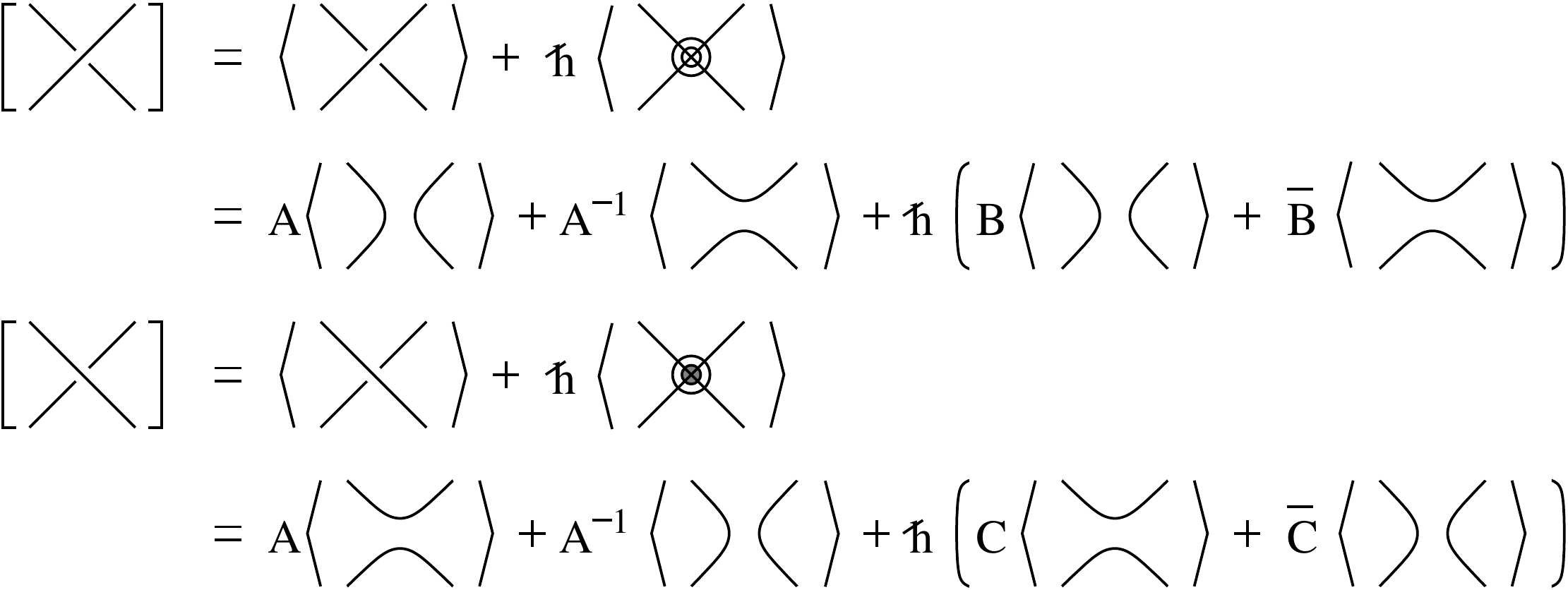}
\end{center}
\caption{}
\label{cocyskein1}
\end{figure}

\begin{lemma}\label{lem:B=C}
 The operators  $ \tilde R = R + \hbar \phi$  and $ \tilde R^{-1} = R^{-1} + \hbar \hat \phi$ satisfy 
 the passcup and passcap identities if and only if 
 $B=C$ and $\bar B = \bar C$.
\end{lemma}

\begin{proof}
A diagrammatic proof is presented in Figure~\ref{passcupcap} for all identities involved.
The identities are (1)=(2), (3)=(4), (5)=(6) and (7)=(8). 
One obtains the stated conditions from these diagrams. 
\end{proof}

\begin{figure}[htb]
\begin{center}
\includegraphics[width=6.5in]{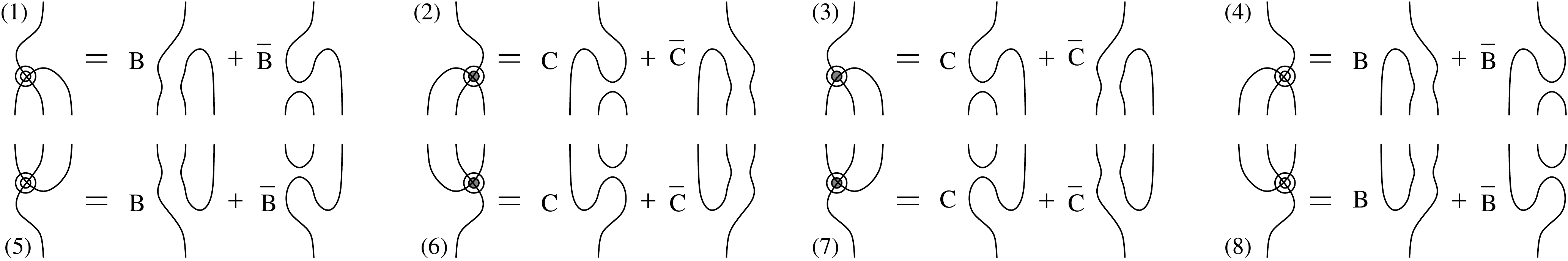}
\end{center}
\caption{}
\label{passcupcap}
\end{figure}

From here on we assume the conditions $B=C$ and $\bar B = \bar C$ in Lemma~\ref{lem:B=C}.

\begin{lemma} \label{lem:inverse}
Under the condition $B=C$ and $\bar B = \bar C$, 
the operators $ \tilde R = R + \hbar \phi$  and $ \tilde R^{-1} = R^{-1} + \hbar \hat \phi$ are inverses to each other if and only if $\bar B = - A^{-2} B$ and $2(A^4 -1) B=0$.
\end{lemma}

\begin{proof}
One of the conditions for the two operators to be inverses to each other is computed under the skein relation in Figure~\ref{inverse0}. 
One of the $\hbar$-degree one part of the inverse conditions is depicted in Figure~\ref{inverse1}
in the left-hand side. 
In the first line in the right-hand side, the defining relations of $\phi$ and $\hat \phi$ are  are applied.
In the second and the third lines, the bracket skein relation is applied. 
By collecting terms for each of $\mathbb 1$ and $\cap \cup$, we obtain
\begin{eqnarray*}
A^{-1} B + A \bar C &=& 0 , \\
AB+A^{-1} \bar C - A^3 \bar B - A^{-3} C &=& 0 .
\end{eqnarray*}
By substituting $B=C$ and $\bar B = \bar C$ we obtain the stated result.
\end{proof}

\begin{figure}[htb]
\begin{center}
\includegraphics[width=3.2in]{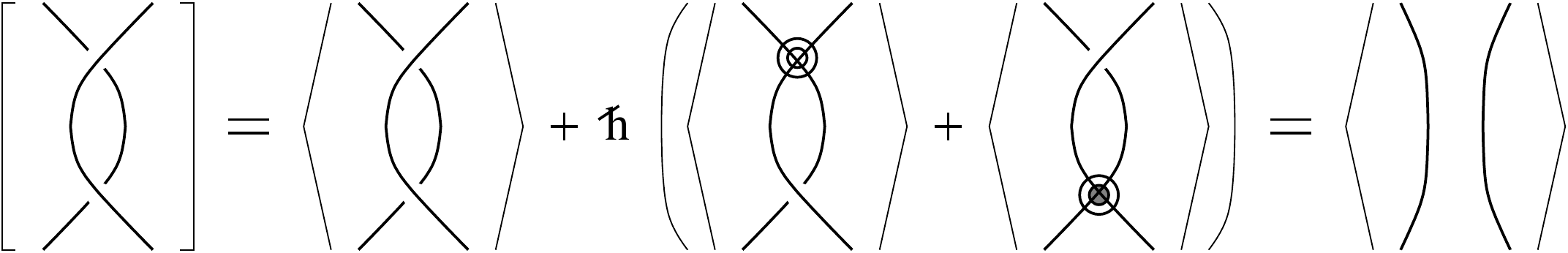}
\end{center}
\caption{}
\label{inverse0}
\end{figure}

\begin{figure}[htb]
\begin{center}
\includegraphics[width=3in]{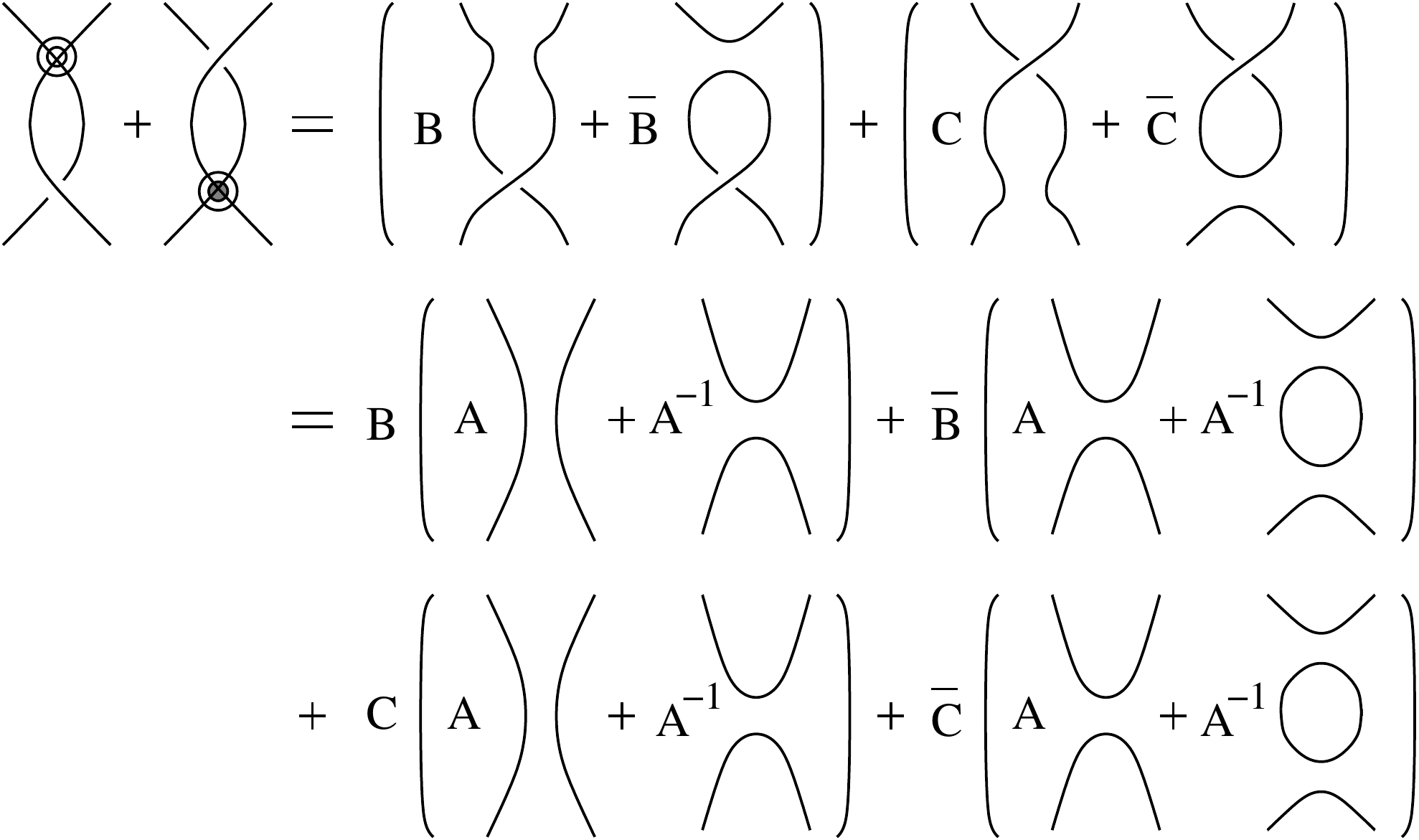}
\end{center}
\caption{}
\label{inverse1}
\end{figure}

\begin{lemma} \label{lem:YB2cocy}
The operators $\tilde R $ and $\tilde R^{-1}$ are YBOs. Equivalently, $\phi$ and $\hat \phi$ are YB $2$-cocycles. 
\end{lemma}

\begin{proof}
The left-hand side of  Figure~\ref{YBcocy} is the first term of the YB 2-cocycle condition in Figure~\ref{YBd2}.
 The top line of the right-hand side of Figure~\ref{YBcocy} shows the expansion of crossings according to the bracket skein relation. In the second line, the skein relation for the YB 2-cocycle $\phi$ is applied.
This computation is carried out for all the eight terms of the 2-cocycle condition in Figure~\ref{YBd2}, 
and set equal to zero. The expressions is written in like terms with respect to the basis elements
of the Temperley-Lieb algebra (e.g. \cite{K&P})  as depicted in Figure~\ref{basis}, and the coefficient of each basis element is set equal to zero. 
The coefficients of the elements (A), (D) and (E) cancel.
Those for (B) and (C) each gives the identical condition $2(A^4-1)B = 0$, which is among the hypotheses for the choice of $A$ and $B$.
The case for $\tilde R^{-1}$ is the same due to the symmetry in the definitions.
\end{proof}

\begin{figure}[htb]
\begin{center}
\includegraphics[width=6.2in]{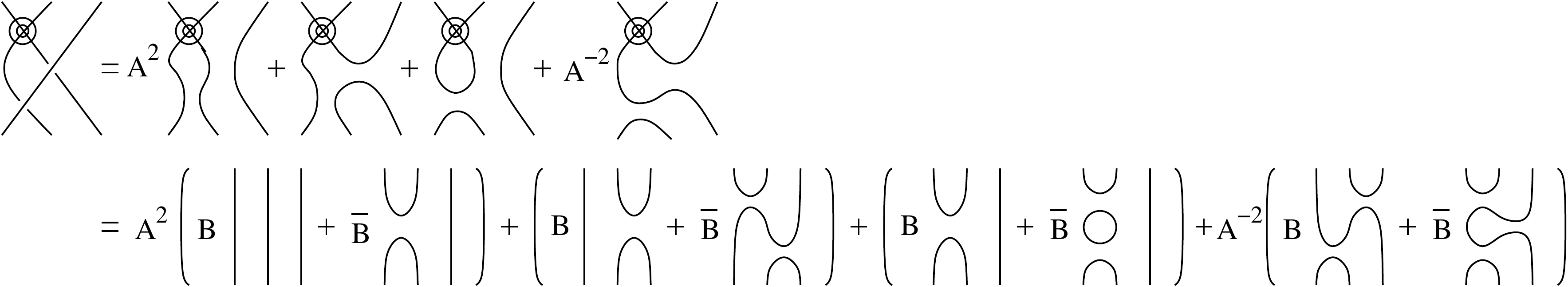}
\end{center}
\caption{}
\label{YBcocy}
\end{figure}

\begin{figure}[htb]
\begin{center}
\includegraphics[width=2.5in]{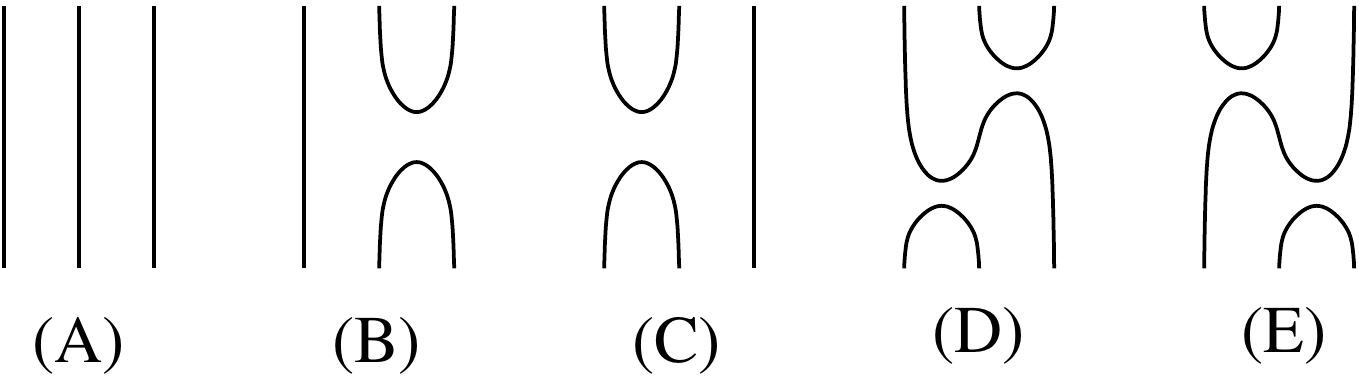}
\end{center}
\caption{}
\label{basis}
\end{figure}

\begin{proposition}\label{prop:allrelations}
 Let $\mathbb k$ be a unital ring that contains an invertible $A$ and $B,\bar B, C, \bar C$. 
The YBOs $\tilde R $ and $\tilde R^{-1}$, together with 
undeformed $\tilde \cup = \cup$ and $\tilde \cap=\cap $ 
define a quantum cocycle knot invariant 
 up to regular isotopy if the following hold:
$$B=C,  \quad \bar B = \bar C, 
\quad \bar B = - A^{-2} B, \quad 2(A^4 -1) B=0. $$ 
\end{proposition}

\begin{proof}
This follows from  Lemma~\ref{lem:B=C}, Lemma~\ref{lem:inverse} and Lemma~\ref{lem:YB2cocy}.  
\end{proof}

\begin{figure}[htb]
\begin{center}
\includegraphics[width=3in]{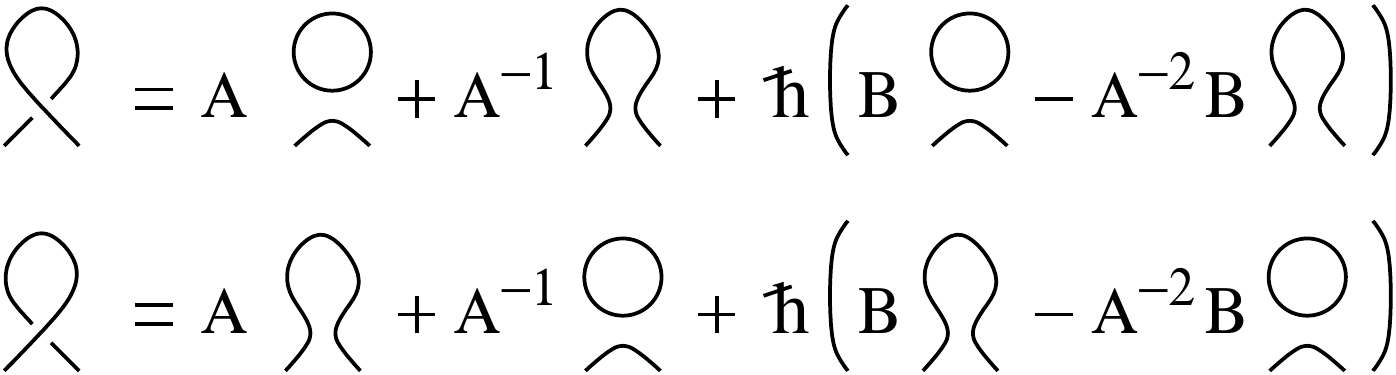}
\end{center}
\caption{}
\label{writhe}
\end{figure}

		\begin{definition}
					{\rm 
					We normalize the regular isotopy invariant defined in Proposition~\ref{prop:allrelations} by dividing it by the value $\delta=- A^2 - A^{-2} $  of the unknot 
					$\begin{tikzpicture}
									\draw (0,0) circle (0.1cm);
							\end{tikzpicture}$, and denote 
					the normalized regular isotopy invariant  by 
					$\Phi^{\rm M}_{(\tilde R, \tilde \cup, \tilde \cap)}(K)$.
										Therefore, we have that $\Phi^{\rm M}_{(\tilde R, \tilde \cup, \tilde \cap)}(\begin{tikzpicture}
							\draw (0,0) circle (0.1cm);
							\end{tikzpicture}) = 1$. 					}
		\end{definition}

The supersctript M stands for maxima/minima.
We note that since $\tilde \cup$ and $\tilde \cap$ are undeformed, the value of  $\begin{tikzpicture}
							\draw (0,0) circle (0.1cm);
							\end{tikzpicture}$ does not contain any terms in $\hbar$, and therefore no new restrictions on $A$ and $B$ are added,   
							as the invertibility of the normalizing factor does not depend on the deformation coefficients.

We examine the effect of writhe to normalize the regular isotopy invariant to obtain an isotopy invariant for oriented knot diagrams, as in the case of the bracket polynomial.
Let $D$ be an oriented diagram.

It has been assumed in this section that knots are unoriented.
To make a normalization by writhe, from here on we pick and fix an orientation. In Figure \ref{writhe} left-hand side, two types of small kinks are depicted. For the top figure, either choice of orientation gives a positive crossing for oriented diagrams, and  a negative crossing for the bottom figure.

\begin{lemma}
Let $D$ be a diagram of an oriented knot obtained from $D'$ by type I Reidemeister move of removing a small
kink 
 as in Figure~\ref{writhe} top (resp. bottom) left.
Set 
\begin{eqnarray*}
w_+:= - A^{3} -  \hbar B A^{-2} (A^{4} + 2 ) , 
  & \ {\rm and} \  & 
w_- := - A^{-3} + \hbar B (A^{-4} + 2 )  .
\end{eqnarray*}
Then we have, respectively, 
$\Phi^{\rm M}_{(\tilde R, \tilde \cup, \tilde \cap)} (D') =  w_+ \Psi^{\rm M}_{(\tilde R, \tilde \cup, \tilde \cap)} (D) $ and  
$\Phi^{\rm M}_{(\tilde R, \tilde \cup, \tilde \cap)}(D') =  w_- \Psi^{\rm M}_{(\tilde R, \tilde \cup, \tilde \cap)} (D). $
\end{lemma}

\begin{proof}
Diagrammatic computations are shown in Figure~\ref{writhe} top and bottom right, respectively.
\end{proof}

We note that the above two coefficients should yield $1$ when multiplied, and indeed we find
\begin{eqnarray*}
w_+ \cdot w_- &=& [ - A^{3} -  \hbar B A^{-2} (A^{4} + 2 ) ]    [ - A^{-3} + \hbar B (A^{-4} + 2 ) ]  \\
&=& 
1 -  \hbar B ( A^3 ( (A^{-4} + 2 )  - A^{-3} A^{-2} (A^{4} + 2 ) ) \\
&=& 
1 - \hbar B [ 2 A^{-5}  ( A^8 - 1 ) ] , 
\end{eqnarray*}
modulo $\hbar^2$, and $2   ( A^8 - 1 ) B $  is divisible by $2(A^4 -1)B$, one of the equalities in Lemma~\ref{prop:allrelations},
 as expected. Hence we treat that $w_+$ and $w_-$ are inverses each other in $\mathbb k$.

By the writhe normalization argument, we obtain the following.

\begin{theorem}\label{thm:bracket}
Let $D$ be an oriented knot diagram with the number of positive (resp. negative)  crossings $p(D)$
(resp. $n(D)$). 
Then 
$$ \Phi^{\rm W}_{(\tilde R, \tilde \cup, \tilde \cap)}  (D) = w_+^{-p(D)} w_-^{-n(D)} \Phi^{\rm M}_{(\tilde R, \tilde \cup, \tilde \cap)} (D) $$ 
is independent of choice of $D$ and indeed is an invariant up to isotopy, thus denoted by 
$ \Phi^{\rm W}_{(\tilde R, \tilde \cup, \tilde \cap)}  (K)  $ and simply  $\Phi^{\rm W} (K)$. 
\end{theorem}

The superscript W stands for writhe-normalized. 
We also use the notation $\Phi^{\rm W} (K)=\Phi^{{\rm W} 0} (K) + \hbar  \Phi^{{\rm W} 1} (K) $.
Since $w_+$ and $w_-$ are inverses of each other, we can also write the prefix powers to be $w_-^{p(D)} w_+^{n(D)}$.

\begin{lemma}\label{lem:Tm}
Let $T_m$ be a $(2,m)$ torus knot or knot  for positive integers $m$,
as depicted in Figure~\ref{Tm}. 
Then we have $ \Phi^{\rm M} (T_1)= w_+$,  $\Phi ^{\rm W}(T_1)=1$ and 
\begin{eqnarray}
 \Phi^{\rm M} (T_{m+1} ) &=&   (A + \hbar B ) \Phi^{\rm M} (T_m ) + (A^{-1} - \hbar A^2 B)  w_+^{m} , \label{eq:Tmbracket} \\
 \Phi^{\rm W} (T_{m+1}) &=&   (A + \hbar B )  w_-  \Phi^{\rm W}(T_m) +  (A^{-1} - \hbar A^2 B) w_- . \label{eq:TmJones}
\end{eqnarray}
\end{lemma}

\begin{proof}
Equation~\eqref{eq:Tmbracket} follows from the definitions and the deformed bracket skein.
Since 
$$ \Phi^{\rm W} (T_m)=w_+^{-m}\Phi^{\rm M}(T_m)=w_-^m \Phi^{\rm M}(T_m), $$
 by multiplying both sides by $w_-^{m+1}$, we obtain Equation~\eqref{eq:TmJones}.
\end{proof}

\begin{example}\label{ex:Tmbracket}
{\rm
Let $T_m$ be a $(2,m)$ torus knot or knot as depicted in Figure~\ref{Tm}. 
Note that all crossings are positive.
By using Equation~\eqref{eq:TmJones}, we obtain 
\begin{eqnarray*}
  \Phi^{\rm W}(T_2) &=&[  (A  + A^{-1}) + \hbar(B + A^2 B) ]  w_- \\
  &=& [ ( A  + A^{-1}) + \hbar (B - A^2B) ] [ - A^{-3} + \hbar B (A^{-4} + 2 ) ]  \\
  &=& -A^{-3}(A  + A^{-1}) + \hbar  B(A^{-5} + 3A^{-1} + 2A). 
\end{eqnarray*}

By Proposition~\ref{prop:allrelations}, the coefficient of $\hbar$ is considered modulo 
$2(A^4-1)$. This is satisfied with $A=\sqrt{-1}$.
In the coefficient of $\hbar B$, 
$$A^{-5} + 3A^{-1} + 2A = - 2 \sqrt{-1}. $$
Thus $\Phi^{{\rm W}} (T_2)$ is nontrivial.

}
\end{example}

	\begin{proposition} 
		For $A=\sqrt{-1}$,  we have $\Phi^{{\rm W}} (T_m) = (2-m) +  \sqrt{-1} (m-1)  (8 - m)  \hbar  B$
		\begin{eqnarray*}
				\Phi^{{\rm W}} (T_m) &=& 
				\begin{cases}
							0 - m\sqrt{-1}B\hbar \ \ \ \ \hspace{10mm} {\rm for}\ m\geq 0\  {\rm even}\\
							\frac{1}{2} - (m-2)\sqrt{-1}B \hbar\ \  \ \ {\rm for}\  m\geq 0\ {\rm odd}
				\end{cases}.
		\end{eqnarray*} 
	\end{proposition}
	
	\begin{proof} 
		From Example~\ref{ex:Tmbracket}, we use the evaluation $A=\sqrt{-1}$.
		In this case the recursive formula of Lemma~\ref{lem:Tm} becomes
		$$ \Phi^{\rm W} (T_{m+1})_{A=\sqrt{-1}}  = ( \sqrt{-1}+ \hbar B )   w_-  \Phi^{\rm W}(T_m)_{A=\sqrt{-1}} +  ( - \sqrt{-1} +   \hbar B) w_- ,$$
		where 
		$w_- := - A^{-3} + \hbar B (A^{-4} + 2 )  = - \sqrt{-1} + 3 \hbar  B  $.
		Hence by simplifying the notation  $\varphi_m:=  \Phi^{\rm W}(T_m)_{A=\sqrt{-1}} $, we obtain 
		\begin{eqnarray*}
			\varphi_{m+1} 
			&=&
			(\sqrt{-1} + \hbar B)(\sqrt{-1}+3\hbar B)\varphi_m + (-\sqrt{-1} + \hbar B)(\sqrt{-1}+3\hbar B)\\
			&=&
			(-1+4\sqrt{-1}B\hbar)\varphi_m + (1-2\sqrt{-1}B\hbar ). 
		\end{eqnarray*}
		
		Set $\varphi_m=a_m + \hbar B b_m$. Then we obtain 
		\begin{eqnarray*}
			a_{m+1} + \hbar B b_{m+1} &=&  (-1+4\sqrt{-1}B\hbar) (a_m + \hbar B b_m) + (1-2\sqrt{-1}B\hbar )\\
			&=& (-a_m +1) + \hbar B (4 \sqrt{-1} a_m - b_m - 2 \sqrt{-1} ),
		\end{eqnarray*}
		so that 
		$$  \begin{cases} a_{m+1} = -a_m +1 , \\
		b_{m+1} = 4 \sqrt{-1} a_m - b_m - 2 \sqrt{-1} , 
		\end{cases}
		$$
		and $a_1=1$, $b_1=0$. 
		It follows that $a_m=0$ for $m$ even, and $a_m = 1$ for $m$ odd, while 
		$b_m = -2m\sqrt{-1}$ for $m$ even, and $b_m = -2(m-2)\sqrt{-1}$ for $m$ odd. The normalizing factor is $\delta = -A^2-A^{-2} = 2$, 
		and dividing $a_m$ and $b_m$ by $2$ we obtain the result. 
	\end{proof}

	\begin{corollary}
		There exist infinitely many knots and links with nontrivial, and distinct, infinitesimal part $\Phi^{{\rm W} 1}$ of the quantum cocycle invariant $\Phi^{{\rm W}}$.
	\end{corollary}

\section{Higher order invariants}\label{sec:higher_def}

We consider now higher order deformations
 of YBOs, where the condition of being enhanced is preserved
 for the trace approach.  It is known that YBOs can be nontrivially deformed to higher order, as seen explicitly in \cite{EZ_perturbative}. Our objective in this section is to characterize when higher order deformations define EYBOs, and how to use such 
 deformed  
 EYBOs to define higher order quantum cocycle invariants.
   Although similar higher order deformations can be formulated
  for the maxima/minima approach with height functions,   
we focus on the trace approach in the remaining sections, as our purpose  of defining this procedure is to give an interpretation of such deformations for the YBOs used for the Jones and Alexander polynomials.

We define the following in a manner similar to the infinitesimal deformation. 
Let 
$$\tilde R_{(n)} = \sum_{i=0}^n \hbar^i\phi_i \in \tilde V_{(n)} = \mathbb k [[\hbar]]/ (\hbar^{n+1} )  \otimes V$$
 be an order $n$ deformation of the EYBO $\phi_0 := R$. 
Then, we say that  $\phi_{n+1}$ is an enhanced order (or degree) $n+1$ deformation if 
$$\tilde R_{(n+1)} := \tilde R_{(n)}  + \hbar^{n+1}\phi_{n+1} =  \sum_{i=0}^{n+1} \hbar^i\phi_i 
\in 
\tilde V_{(n+1)} = \mathbb k [[\hbar]]/ (\hbar^{n+2} ) \otimes V   $$
 is an EYBO.

The following theorem provides conditions for $(\phi_{n+1}, \mu_{n+1})$ 
to extend deformation $\tilde R_{(n)}$ to $\tilde R_{(n+1)}$ as an EYBO. 
 
 We recall the following notation from Definition~\ref{def:gamma}.
 \begin{eqnarray*}
					\hat  \Gamma_n^k = \{(j_1,\ldots,j_k )\ |\  j_1+\cdots +j_k= n, \  j_\ell \geq 0\ {\rm and}\  j_\ell \neq n,\ {\rm for\ all}\ \ell=1,\ldots,k \}.
\end{eqnarray*}
Recall also, when $k$ is apparent, we use $ \hat  \Gamma_n=\hat  \Gamma_n^k$. 

		We also define, following \cite{SZ-YBH}, the quantity
	\begin{eqnarray*}
		\Theta_{n+1} 
		&:=& \sum_{(i,j,k)\in \hat\Gamma_{n+1}} [(\phi_i\otimes \mathbb 1)(\mathbb 1\otimes \phi_j)(\phi_k\otimes \mathbb 1)- (\phi_k\otimes \mathbb 1)(\phi_i\otimes \mathbb 1)(\mathbb 1\otimes \phi_j)]
	\end{eqnarray*}

For  $\hat \phi_i$ indicating  the $i^{\rm th}$ component of the inverse of $\tilde R_{(n)} + \hbar^{n+1}\phi_{n+1}$ (modulo $\hbar^{n+2}$), we use the same notation with hat, 	$\hat \Theta_{n+1} $, to denote the same expression as above with  $\hat \phi_i$ on the right-hand side.

\begin{theorem}\label{thm:higher_YB_inv}
		Let $R$ be a YBO, and let $\phi_1, \ldots, \phi_n$, $\mu_1, \ldots, \mu_n$, $\alpha, \beta$ be such that  $S = (\tilde R_{(n)}, \alpha, \beta, \tilde \mu)$ is an EYBO, where $\tilde R_{(n)}= \sum_{i=0}^n \hbar^i\phi_i$ and $\tilde \mu_n = \sum_{i=0}^n \hbar^i\mu_i$. Then the obstruction for $(\phi_{n+1}, \mu_{n+1})$ to be an enhanced order $n+1$ deformation of $R$ is the following system of equations
		\begin{eqnarray}\label{eqn:higher_deform0}
			\delta^2_{\rm YB} \phi_{n+1} + \Theta_{n+1} = 0, \quad    \delta^2_{\rm YB} \hat  \phi_{n+1} + \hat \Theta_{n+1} = 0 
			\end{eqnarray}
			and
		\begin{eqnarray}\label{eqn:higher_deform}
				\begin{cases}
											\sum_{(i,j,k)\in \hat \Gamma_{n+1}} [(\mu_i\otimes \mu_j)\phi_k - \phi_i(\mu_j\otimes \mu_k)] = 0 , \\
											 			\sum_{(i,j,k)\in \hat \Gamma_{n+1}} [(\mu_i\otimes \mu_j)\hat \phi_k -\hat  \phi_i(\mu_j\otimes \mu_k)] = 0 ,  \\
				\sum_{(i,j,k)\in \hat  \Gamma_{n+1}} \tr(\phi_i(\mu_j\otimes \mu_k)) = \alpha\beta\mu_{n+1} , \\
						\sum_{(i,j,k)\in  \hat \Gamma_{n+1}} \tr(\hat \phi_i(\mu_j\otimes \mu_k)) = \alpha^{ -1}  \beta\mu_{n+1} .
				\end{cases}
		\end{eqnarray}
	\end{theorem}
		
\begin{proof} 
	First we  show that enhanced higher order deformations are equivalent to the system~\eqref{eqn:higher_deform0} and \eqref{eqn:higher_deform}. It is a known result, see for instance \cite{SZ-YBH}, that the obstruction for $\phi_{n+1}$ to be an order $n+1$ YB deformation of $\tilde R_{(n)}$ is $\delta^2_{\rm YB} \phi_{n+1} + \Theta_{n+1} = 0$, where the left-hand side lies in the third cohomology group of $R$. Hence 
		$\phi_{n+1}$ extends $\tilde R_{(n)}$ to a YBO if and only if the first equality of~\eqref{eqn:higher_deform0} holds.

		Commutativity of $\tilde R_{(n+1)}= \tilde R_{(n)}+ \hbar^{n+1}\phi_{n+1}$ with the deformed 
		$\tilde  \mu_{n+1} = \tilde \mu_{n} + \hbar^{n+1}\mu_{n+1}$ means that Equation~\eqref{eqn:mu_comm} holds with $\tilde \mu_{(n+1)}$ and $\tilde R_{(n+1)}$. Taking terms of degree $n+1$ in $\hbar$ gives the first  
		equation in~\eqref{eqn:higher_deform}. Equation~\eqref{eqn:partial_tr} with positive sign gives the 
		second 
		equation of~\eqref{eqn:higher_deform} when considering terms of degree $n+1$ in $\hbar$. Similarly, we obtain the third 
		 equation of the system from Equation~\eqref{eqn:partial_tr} with negative sign, where the existence of the inverse of $\tilde R_{(n+1)} $ has been shown in the first part of the proof. This completes the first statement of the theorem. 

		Next  we  
		show that $\tilde R_{(n)}  = \sum_{i=0}^n \hbar^i\phi_i$ is invertible modulo $\hbar^n$ for any $n$. Let us set 
	 $\tilde R^{-1}_{(n)}  = \sum_{i=0}^n \hbar^i  \hat \phi_i$. The fact that  $\tilde R^{-1}_{(n)}$ is an inverse to $\tilde R_{(n)}$ is equivalent to the following two systems. 
		\begin{eqnarray}\label{eqn:system_inverse_right}
	\hat \phi_0 = \phi_0^{-1}, \ 
					\ldots, \ 
					\hat\phi_k = -\phi_0^{-1}\sum_{i=1}^k \phi_i\hat\phi_{k-i},
					\  \ldots, \ 
										\hat\phi_n = -\phi_0^{-1}\sum_{i=1}^n \phi_i\hat\phi_{n-i}
\end{eqnarray}
		and 
		\begin{eqnarray}\label{eqn:system_inverse_left}
\hat \phi_0 = \phi_0^{-1}, \ \ldots, \  	\hat\phi_k = -\sum_{i=1}^k \phi_i\hat\phi_{k-i}\cdot \phi_0^{-1}, \ \ldots, \  
 	\hat\phi_n = -\sum_{i=1}^n \phi_i\hat\phi_{n-i} \cdot \phi_0^{-1}
		\end{eqnarray}
		where \eqref{eqn:system_inverse_right} gives right-invertibility, and \eqref{eqn:system_inverse_left} corresponds to left invertiblity. 
		Additionally, we need to verify that the solutions obtained by the two systems coincide. Let $\phi_{i_1}, \ldots, \phi_{i_k}$ be a list of elements appearing in $\tilde R_{(n)}$, not necessarily mutually distinct. We define the symmetrized expression
		\begin{eqnarray}
				(-1)^n\mathbb S(\phi_{i_1}|\ldots|\phi_{i_k}) = \sum_{\sigma\in \Sigma_k} \frac{1}{q_1!\cdots q_t!} \phi_0^{-1}\phi_{\sigma(i_1)}\phi_0^{-1} \cdots \phi_0^{-1} \phi_{\sigma(i_k)}\phi_0^{-1},
		\end{eqnarray}
		where $q_1, \ldots, q_t$ indicate the number of repeated indices among the $i_r$. Then, we have for example 
		$$\mathbb S(\phi_1| \cdots | \phi_1) = - \phi_0^{-1}\phi_1\phi_0^{-1} \cdots \phi_0^{-1} \phi_1\phi_0^{-1}$$
		when $\phi_1$ is repeated $2n+1$ times, 
		$$\mathbb S(\phi_1| \cdots | \phi_1) = \phi_0^{-1}\phi_1\phi_0^{-1} \cdots \phi_0^{-1} \phi_1\phi_0^{-1}$$
		 when $\phi_1$ is repeated $2n$ times. Similarly, by a direct computation we have 
		 $$\mathbb S(\phi_1|\phi_1|\phi_2) = - \phi_0^{-1}\phi_1\phi_0^{-1}\phi_1\phi_0^{-1}\phi_2\phi_0^{-1} - \phi_0^{-1}\phi_1\phi_0^{-1}\phi_2\phi_0^{-1}\phi_1\phi_0^{-1} - \phi_0^{-1}\phi_2\phi_0^{-1}\phi_1\phi_0^{-1}\phi_1\phi_0^{-1}$$
		  and so forth. We  
		  now show  that the inverse of $\tilde R_{(n)}$ is given by 
		\begin{eqnarray}\label{eqn:phin}
				\hat \phi_n = \sum \mathbb S(\phi_{i_1}| \cdots |\phi_{i_k}),
		\end{eqnarray}
		where the sum is taken over all the possible choices of indices $i_1, \ldots, i_k$ such that $i_1 \leq i_2 \leq \cdots \leq i_k$ and $\sum_{d=1}^k i_d = n$. We proceed by induction to show that $\hat \phi_n$ chosen as in Equation~\eqref{eqn:phin} gives a solution of the system of equations~\eqref{eqn:system_inverse_left}. 
		
		The case $n=0$ is trivial, since $\hat \phi_0 = \phi_0^{-1}$, and case $n=1$ was seen to produce $\hat \phi_1 = \phi_0^{-1}\phi_1\phi_0^{-1}$, which coincides with $\mathbb S(\phi_1)$, and therefore with Equation~\eqref{eqn:phin} since there is no sum. 
		
		Suppose now that the inverse for $\tilde R_{(n)}$ has been determined when there are $n$ terms, with $n>1$. When we consider~\eqref{eqn:system_inverse_left} for $\tilde R_{(n+1)} $. 
		 The first $n$ equations are identical to the case of length $n$. Therefore, they are given by Equation~\eqref{eqn:phin} by the inductive assumption. The $(n+1)^{\rm th}$ equation is given by 
		$$\hat \phi_{n+1} = - \phi_0^{-1} \sum_{i=1}^{n+1} \phi_i \hat \phi_{n+1-i} = - \phi_0^{-1} \sum_{i=1}^{n+1} \phi_i\sum \mathbb S(\phi_{i_1}|\cdots |\phi_{i_k}), $$
		where the second sum is defined as in Equation~\eqref{eqn:phin}. A direct inspection shows that the sum in $\hat \phi_{n+1}$ gives all the possible product of terms of type $\phi_0^{-1}\phi_{i_1} \phi_0^{-1} \cdots \phi_0^{-1}\phi_{i_r}\phi_0^{-1}$ where the indices $i_1, \ldots, i_r$ add up to $n+1$. This completes the inductive step. The case of~\eqref{eqn:system_inverse_right} is analogous. This procedure also shows that both right and left inverses coincide, therefore implying that $\tilde \phi$ is invertible for any $n$. This implies that the system of equations~\eqref{eqn:higher_deform} is well posed.
		\end{proof}

\begin{definition}
	{\rm
			We say that 
			  $\tilde R_{(1)}  = R + \hbar \phi_1$ 
			is an {\rm integrable enhanced deformation} when the obstruction to deforming $\tilde R_{(n)} = R + \sum_{i=1}^n\hbar^i\phi_i$ to degree $n+1$, $\tilde R_{(n+1)} = R + \sum_{i=1}^{n+1}\hbar^i\phi_i$  vanishes for all $n$. In other words, if for all $n=1, 2, \ldots,  $ 
			 there exists 
			$\phi_{n+1}$ such that $\tilde R_{(n)} = R + \sum_{i=1}^{n+1}\hbar^i\phi_i$ is an EYBO over 
			$\tilde V_{(n)} = \mathbb k[[\hbar]]/(\hbar^{n+1})\otimes V$. In this case, we formally write $\tilde R = \sum_{i=0}^\infty \hbar^i\phi_i$, where we set $R = \phi_0$.  
	}
\end{definition}

\begin{theorem}\label{thm:higher}
Let a knot $K$ be given by a diagram  that is in a closed braid form of an $m$-braid $b_m$. Suppose that $\tilde R_{(n)}$ is an EYBO with inverse $\tilde R^{-1}_{(n)} $
for all $n=1, \ldots, N$. 

Let $\Psi_{ (\tilde R_{(N)}, \tilde R^{-1}_{(N)} ) } (b_m) $ 
 be the operator defined from deformed YBOs $\tilde R^{\pm}_{(n)} $. Then we have that
\begin{eqnarray}
				\Phi_{ (\tilde R_{(N)}, \tilde R^{-1}_{(N)} )  } (b_m) = 	\sum_{(i,j_1,\ldots,j_m)\in \Gamma^m_n}  \alpha^{-w(b_m)}\beta^{-m}\tr(\Psi_{ (\tilde R_{(N)}, \tilde R^{-1}_{(N)}) } (b_m)(\mu_{j_1}\otimes \cdots \otimes \mu_{j_m}))
		\end{eqnarray}
is invariant under Markov moves, and therefore defines a knot invariant  $\Phi_{ (\tilde R_{(N)}, \tilde R^{-1}_{(N)} )  } (K)$.

Write $	\Phi_{ (\tilde R_{(N)}, \tilde R^{-1}_{(N)} )  }  (K) = 	\sum_{i=0}^n \hbar^i \Phi^i_{ (\tilde R_{(N)}, \tilde R^{-1}_{(N)} )  } (K) $, then we have
	\begin{eqnarray}
			\Phi^i_{ (\tilde R_{(N)}, \tilde R^{-1}_{(N)} )  } (K)=	\sum_{(i,j_1,\ldots, j_m)\in \Gamma^{m+1}_k} \alpha^{-w(b_m)}\beta^{-m}\tr(\Psi^i_{ (\tilde R_{(N)}, \tilde R^{-1}_{(N)} )  } (b_m)(\mu_{j_1}\otimes \cdots \otimes \mu_{j_m}))
		\end{eqnarray}

\end{theorem}

\begin{proof}
We define the quantum invariant of Turaev \cite{Tur} associated to the EYBO $\tilde \phi$. Since $\tilde \phi$ is graded by the formal variable $\hbar$, it follows that each coefficient of the monomials in $\hbar$ in the quantum invariant is invariant  under Markov moves. Considering the degree $n$ component in $\hbar$ we find the term $\sum_{(i,j_1,\ldots, j_m)\in \Gamma^{m+1}_k} \alpha^{-w(b_m)}\beta^{-m}\tr(\Psi^i_\phi(b_m)(\mu_{j_1}\otimes \cdots \otimes \mu_{j_m}))$, which is therefore an 
 invariant of the knot 
$K$.			This completes the proof.
\end{proof}

 Although it is desirable to have cencrete examples of higher order invariants as defined above, it is out of the scope of this paper, and the purpose of this section is to generalize this further to Laurent polynomial deformation to 
give an interpretation of the Jones and Alexander polynomials as such quantum cocycle invariant in the next section.

\section{The Jones and Alexander polynomials as higher order  invariants}\label{sec:Jones_Alex}
		
In this section we discuss how to interpret the Jones and Alexander polynomials as quantum cocycle invariants. We need a slight generalization of the framework considered up to now, before proceeding with the main results of the section. 

Note that  the YB cohomology used in this article does not depend on the invertibility of the YBO, as opposed to that used in \cite{Eis,Eis1}, for example. Therefore, it still makes sense to discuss YB cocycles of pre-YB operators, i.e. operators that satisty the YBE, but are not necessarily invertible. While we have assumed invertibility so far, in this section we will 
not make this assumption.

\begin{definition}\label{def:Laurent}
{\rm
Let 
$$\tilde R_{(n)}=\sum_{i=0}^n \hbar^i \phi_i  \in \tilde V_{(n)} = \mathbb k[[\hbar]]/(\hbar^{ n+1 } ) \otimes  V
\quad {\rm  and} \quad
\tilde R^{-1}_{(n)} = \sum_{i=0}^n \hat\hbar^i \hat \phi_i  \in \hat V_{(n)} = \mathbb k[[\hat \hbar]]/(\hat \hbar^{ n+1}  ) \otimes  V$$
be  pre-YBOs for all $n=0, \ldots, N$.
Regard $\tilde V_{(N)}$ and $\hat V_{(N)}$ as submodules of $$\mathcal V = \mathbb k [[\hbar, \hat \hbar ]]/ (\hbar \hat \hbar -1) .$$ 
If $\tilde R_{(N)}$ and $\tilde R^{-1}_{(N)}$ satisfy the YBE in $\mathcal V$, and are inverse to each other in $\mathcal V$, then we say that 
$(\tilde R_{(N)}, \tilde R^{-1}_{(N)})$ is a YBO obtained by {\it Laurent deformation } of pre-YBOs $(R, R^{-1})=(\phi_0, \hat \phi_0)$, of order $N$. 
We do not exclude the case where $N = \infty$, and $\tilde R_{(\infty)}=\sum_{i=0}^\infty \hbar \phi_i$, $\tilde R^{-1}_{(\infty)} = \sum_{i=0}^\infty \hat\hbar^i \hat \phi_i$. 
}
\end{definition}

We further elucidate the situation in Definition~\ref{def:Laurent}. For finite $N$, each $\tilde R_{(n)}$ is a pre-YBO modulo factors of degree $\hbar^{n+1}$, for $n=0, \ldots, N-1$. However, 
when considering $\tilde  R_{(N)}$, 
 it is assumed in the definition that 
this operator satisfies the YBE exactly, in the sense that the YBE is satisfied even without quotienting modulo $\hbar^{N+1}$. 
A similar assumption is made for $R^{-1}_{(n)}$. Moreover, the operators $\tilde R_{(N)}$ and $\tilde  R^{-1}_{(N)}$ are inverse to each others when the condition $\hat \hbar = \hbar^{-1}$ is imposed. When $N=\infty$, we have that $\tilde  R_{(n)}$ and $\tilde  R^{-1}_{(n)}$ are pre-YBOs for each $n\in \mathbb N$, but the full power series satisfy the  YBE exactly, and they are inverse to each other. When deformations are integrable, we can automatically find such $\tilde  R_{(\infty)}$ and $\tilde  R^{-1}_{(\infty)}$. We will provide two examples of finite order Laurent deformations that correspond to the Jones and Alexander polynomials below. They are both Laurent deformations of order $2$, and we will see that while the first order deformations are only pre-YBOs, hence they are invertible only modulo degrees of order $2$, the second order deformations satisfy the YBE exactly, and are invertible exactly. 
 In the following we use the notations $\Theta_{n+1}, \hat \Theta_{n+1}$ defined above Theorem~\ref{thm:higher_YB_inv}.

\begin{lemma}
Suppose that 
$\tilde R_{(n)}$ and $R^{-1}_{(n)}$ 
are pre-YBOs, and that $\phi_{n+1}$ and
$\hat \phi_{n+1} $ satisfy Equation~\ref{eqn:higher_deform0}:
	$$\delta^2_{\rm YB}\phi_{n+1} + \Theta_{n+1} = 0 \quad {\rm and } \quad \delta^2_{\rm YB}\hat \phi_{n+1} + \hat \Theta_{n+1} = 0$$
respectively, for all $n=1, \ldots, n$, and that $\tilde R_{(n+1)}$ and $R^{-1}_{(n+1)}$ are inverses to each other in $\mathcal V$,
then $( \tilde R_{(n+1)}, \tilde R^{-1}_{(n+1)}  )$ is a Laurent deformation of $(\phi_0, \hat \phi_0)$.
\end{lemma}

\begin{proof}
In the proof of the first statement of Theorem~\ref{thm:higher_YB_inv}, it was not assumed that $\tilde R_{(n)}$ is 
invertible for any $n$, and that Equation~\eqref{eqn:higher_deform0} is sufficient for $\tilde R_{(n)}$
to be a YBO. This also holds for $R^{-1}_{(n)}$ as well. 
\end{proof}

\begin{definition}\label{def:ELaurent}
{\rm

Let $(\tilde R_{(N)}, \tilde R^{-1}_{(N)})$ be a Laurent deformation of $(\phi_0, \hat \phi_0)$ as in Definition~\ref{def:Laurent}.
If $\tilde R_{(N)}$ and $\tilde R^{-1}_{(N)}$ are both enhanced YBO, then we call the pair 
$(\tilde R_{(N)}, {\tilde R_{(N)}}^{-1} )$ a {\it Laurent enhanced YBO (LEYBO)}.
}
\end{definition}

\begin{lemma}\label{lem:LEYBOinv}
Let 
 $(\tilde R_{(N)}, \tilde R^{-1}_{(N)})$ be a Laurent enhanced YBO. 
 Let $b_m$ be an $m$-braid whose closure represents a knot $K$. 
Let $\Psi_{(\tilde R_{(N)}, \tilde R^{-1}_{(N)}) } (b_m)$ be the braid representation defined from 
 $(\tilde R_{(N)}, \tilde R^{-1}_{(N)})$.
Let 
\begin{eqnarray*}
				\Phi_{(\tilde R_{(N)}, \tilde R^{-1}_{(N)})}(b_m) = 	\sum_{(i,j_1,\ldots,j_m)\in \Gamma^m_n}  \alpha^{-w(b_m)}\beta^{-m}\tr(\Psi_{(\tilde R_{(N)}, \tilde R^{-1}_{(N)})  }(b_m)(\mu_{j_1}\otimes \cdots \otimes \mu_{j_m}))
		\end{eqnarray*}
be as in Theorem~\ref{thm:higher}.
 Then $\Phi_{(\tilde R_{(N)}, \tilde R^{-1}_{(N)}) }(b_m) $ defines a knot invariant $\Phi_{(\tilde R_{(N)}, \tilde R^{-1}_{(N)}) }(K) $,  called a {\it quantum cocycle invariant from LEYBO}.
\end{lemma}

\begin{proof}
The proof of Theorem~\ref{thm:higher} applies without assumption of invertibility of
$\tilde R_{(n)}$ and $\tilde R^{-1}_{(n)}$ for $n < N$, and only requires invertibility of 
$\tilde R_{(N)}$ and $\tilde R^{-1}_{(N)}$. 
\end{proof}

\begin{theorem}\label{thm:Jones}
			The Jones polynomial is a 
			quantum cocycle invariant from LEYBO. 
			Moreover, let $D$ be a closed braid diagram of the knot $ K$ on $n$ strings, with $m^+$ positive  and $m^-$ negative crossings.
			Then, we have
			\begin{eqnarray*}
					(\hbar+\hbar^{-1})J(K) = \hbar^{m^+-m^--n} 
					\Phi_{(\tilde R_{(N)}, \tilde R^{-1}_{(N)}) }(K)
			\end{eqnarray*}
			for some 
			$\tilde R_{(N)}$ and $\tilde R^{-1}_{(N)}$. 
\end{theorem}
\begin{proof} 
		We explicitly construct  $\tilde R_{(N)}$ and $\tilde R^{-1}_{(N)}$. 
		
		Recall (e.g.  \cite{Ohtsuki})  that the Jones polynomial (up to a factor of $t^{\frac{1}{2}}+t^{-\frac{1}{2}}$)
		can be obtained as a trace invariant of an appropriate EYBO $(J,\mu, 1, 1)$. We use this EYBO and rewrite it using the formalism developed in this article. By  setting  $t^{\frac{1}{2}} = \hbar$ we have
			\begin{eqnarray*}
				J =  \begin{bmatrix}
					t^{\frac{1}{2}} & 0 & 0 & 0\\
					0 & 0 & t & 0\\
					0 & t & t^{\frac{1}{2}}-t^{\frac{3}{2}} & 0\\
					0 & 0 & 0 & t^{\frac{1}{2}}
				\end{bmatrix}
				= \hbar
				\begin{bmatrix}
						1 & 0 & 0 & 0\\
						0 & 0 & \hbar & 0\\
						0 & \hbar & 1-\hbar^2 & 0\\
						0 & 0 & 0 & 1
				\end{bmatrix}.
		\end{eqnarray*}
		 This  
		 form of $J$ can be rewritten as
		 $$J=J_0 + \hbar J_1 + \hbar^2 J_2, $$
where 	
	\begin{eqnarray*}
		J_0 = \begin{bmatrix}
		1 & 0 & 0 & 0\\
		0 & 0 & 0 & 0\\
		0 & 0 & 1 & 0\\
		0 & 0 & 0 & 1
		\end{bmatrix} ,
		\quad
		J_1= \begin{bmatrix}
		0 & 0 & 0 & 0\\
		0 & 0 & 1 & 0\\
		0 & 1 & 0 & 0\\
		0 & 0 & 0 & 0
		\end{bmatrix} ,
		\quad 
		J_2=		\begin{bmatrix}
	0 & 0 & 0 & 0\\
			0 & 0 & 0 & 0\\
			0 & 0 & -1 & 0\\
			0 & 0 & 0 & 0
		\end{bmatrix}.
	\end{eqnarray*}

 By  computer calculations, we find 
$J_0$ 	is a pre-YBO, i.e. it satisfies the YBE but it is 
  not invertible, 
and $J_1$ is a YB $2$-cocycle.
However, since the YBE for $J$ holds exactly, as opposed as being up to higher orders of $\hbar$ than $2$, it means that the deformation is integrable by adding zero maps of all orders. 
	
	To find an inverse, we proceed analogously by using the matrix
	\begin{eqnarray*}
			J^{-1} = \begin{bmatrix}
				1 & 0 & 0 & 0\\
				0 & 1-\hbar^{-2} & \hbar^{-1} & 0\\
				0 & \hbar^{-1} & 0 & 0\\
				0 & 0 & 0 & 1
			\end{bmatrix},
	\end{eqnarray*}
	and decompose this in terms of $\hbar^{-1}$, denoting the matrices by $\hat J_i$, $i=0,1,2$. The same discussion as for $J$ applies here too, mutatis mutandis.

	Next, we have
	\begin{eqnarray*}
			\mu = \begin{bmatrix}
				t^{-\frac{1}{2}} & 0\\
				0 & t^{\frac{1}{2}}
			\end{bmatrix}
					= \hbar^{-1} \begin{bmatrix}
						1 & 0\\
						0 & \hbar^2
					\end{bmatrix}.
	\end{eqnarray*}
	It follows that 
	\begin{eqnarray*}
			\hbar \mu = \mu_0 + \mu_2 = 
			\begin{bmatrix}
			1 & 0\\
			0 & 0
			\end{bmatrix}
			+ 
			\hbar^{-1} \begin{bmatrix}
			0 & 0\\
			0 & \hbar^2
			\end{bmatrix}.
	\end{eqnarray*}
	
	To show that 
	\begin{eqnarray*}
		(\hbar+\hbar^{-1})J( K) = \hbar^{m^+-m^--n} \Phi_{(\tilde R_{(N)}, {\tilde R_{(N)}}^{-1} )}(  K),
	\end{eqnarray*}
	we set $\phi = \sum_{i=0}^2 \hbar^i J_i$ and $\hat \phi = \sum_{i=0}^2 \hbar^i \hat J_i$. Then, 
	$\Phi_{(\tilde R_{(N)}, \tilde R_{(N)}^{-1} ) } (  K)$ differs from the Jones polynomial $J(  K)$ by only a multiplying power of $\hbar$, which was initially factored out when constructing the matrices $J$ and $J^{-1}$. Since $J$ needs to be multiplied by $\hbar$ and $J^{-1}$ by $\hbar^{-1}$, it follows that $\hbar^{m^+-m^-}$ corrects this deficiency and gives $J( K)$ as needed. Lastly, each $\mu$ contributes with a factor $\hbar^{-1}$, and there are $n$ of them since $K$ is obtained as the closure of a braid on $n$ strings.  
\end{proof}

Although the trace is taken only in the factors $2, \ldots, m$ leaving the first factor in the following theorem, 
so that it is not fully LEYBO as defined, it provides a similar description for the Alexander polynomial, and 
we make the following statement.

\begin{theorem}\label{thm:Alex}
	The Alexander polynomial $\Delta( K)$ of a  
	knot $ K$ is a  
	quantum cocycle invariant 
	in the following sense. 
		Let $K$ be represented by a closed braid form of an $m$-braid $b_m$. 
		Assume that $b_m$ has $m^+$ positive  and $m^-$ negative crossings. Then, we have
	\begin{eqnarray*}
		\Delta (K) \cdot \mathbb 1 = \hbar^{-m^++m^-+m-1} \tr_{2,\ldots, m}
		(\Psi_{(\tilde R_{(N)}, \tilde R^{-1} _{(N)} ) }(b_m) (\mathbb 1\otimes \mu^{\otimes (m-1)}) ) ,
	\end{eqnarray*}
	for some  $\tilde R_{(N)}$ and $\tilde R_{(N)}^{-1}$, where $\tr_{2,\ldots, m}$ is the trace over the $2, 3, \ldots, m$ factors 
	 of ${\mathcal V}$.
\end{theorem}

\begin{proof} 
			The proof is analogous to the case of the Jones polynomial above.  We have, from \cite{Ohtsuki}, the equality 
			\begin{eqnarray*}
					\Delta(K)\cdot \mathbb 1 = \tr_{2,\ldots, m}(\psi_{\hat \Delta}(b_m) (\mathbb 1\otimes \mu^{\otimes (m-1)})), 
			\end{eqnarray*}
			where $\psi_{\hat \Delta}(b_m)$ is the operator obtained from $b_m$ by replacing crossings with the matrix 
			\begin{eqnarray*}
				\hat \Delta = \begin{bmatrix}
					t^{-\frac{1}{2}} & 0 & 0 & 0\\
					0 & 0 & 1 & 0\\
					0 & 1 & t^{-\frac{1}{2}}-t^{\frac{1}{2}} & 0\\
					0 & 0 & 0 & -t^{\frac{1}{2}}
				\end{bmatrix}, 
			\end{eqnarray*}
			and 
			\begin{eqnarray*}
					\mu = \begin{bmatrix}
						t^{\frac{1}{2}} & 0\\
						0 & -t^{\frac{1}{2}}
					\end{bmatrix}
					=
				\hbar  \begin{bmatrix}
						1 & 0\\
						0 & -1
					\end{bmatrix}.
			\end{eqnarray*}
			We factor out a term of $\hbar^{-1} = t^{-\frac{1}{2}}$ from the matrix $\hat \Delta$ to obtain
			\begin{eqnarray*}
				\Delta = \begin{bmatrix}
					1 & 0 & 0 & 0\\
					0 & 0 & \hbar & 0\\
					0 & \hbar & 1-\hbar^2 & 0\\
					0 & 0 & 0 & -\hbar^2
				\end{bmatrix}.
			\end{eqnarray*}
							Its inverse is given by 
					\begin{eqnarray*}
						\Delta^{-1} = \begin{bmatrix}
							1 & 0 & 0 & 0\\
							0 & 1-\hbar^{-2} & \hbar^{-1} & 0\\
							0 & \hbar^{-1} & 0 & 0\\
							0 & 0 & 0 & -\hbar^{-2}
						\end{bmatrix}.
					\end{eqnarray*}

			 As in the case of the Jones polynomial, we have $\Delta_i$ with $i=0,1,2$ given by
					\begin{eqnarray*}
						\Delta_0 = \begin{bmatrix}
							1 & 0 & 0 & 0\\
							0 & 0 & 0 & 0\\
							0 & 0 & 1 & 0\\
							0 & 0 & 0 & 0
						\end{bmatrix}, 
						\quad
						\Delta_1 = \begin{bmatrix}
							0 & 0 & 0 & 0\\
							0 & 0 & \hbar & 0\\
							0 & \hbar & 0 & 0\\
							0 & 0 & 0 & 0
						\end{bmatrix},
						\quad
						\Delta_2 = \begin{bmatrix}
							0 & 0 & 0 & 0\\
							0 & 0 & 0 & 0\\
							0 & 0 & -\hbar^2 & 0\\
							0 & 0 & 0 & -\hbar^2
						\end{bmatrix}.
					\end{eqnarray*}
			As in the case of the Jones polynomial, computer calculations show that $\Delta_0$ is a pre-YBO and $\Delta_1$ is a YB 2-cocycle.
			
			Since $\hbar^{-1}$ is factored out for $\Delta$, and $\hbar$ is factored out for $\Delta^{-1}$  (i.e., $\hbar \Delta^{-1}= \tilde{\Delta}^{-1}$), we find that the generalized quantum cocycle invariant needs to be multiplied by $\hbar^{-m^++m^-}$ to give the Alexander polynomial. The map $\mu$ introduces a factor of $\hbar^{m-1}$, one for each copy of $\mu$. 
			The details are analogous to the case of the Jones polynomial in Theorem~\ref{thm:Jones}, except the fact that the trace is a partial trace on the entries $2, \ldots, m$ of $V^{\otimes m}$. 
\end{proof}

\noindent
{\bf Acknowledgement.} The authors are grateful to the referee
for reading the manuscript very carefully and providing numerous comments in depth that helped  significantly improve the paper.

\end{document}